\documentclass[]{interact}

\usepackage{epstopdf}
\usepackage[caption=false]{subfig}

\usepackage[numbers,sort&compress]{natbib}
\bibpunct[, ]{[}{]}{,}{n}{,}{,}
\makeatletter
\def\NAT@def@citea{\def\@citea{\NAT@separator}}
\makeatother

\usepackage{amsmath, amssymb, amsfonts, latexsym, xcolor, mathrsfs, longtable, enumitem,xypic,hyperref}

\theoremstyle{plain}
\newtheorem{theorem}{Theorem}[section]
\newtheorem{lemma}{Lemma}[section]
\newtheorem{corollary}{Corollary}[section]
\newtheorem{proposition}{Proposition}[section]

\theoremstyle{definition}
\newtheorem{definition}{Definition}[section]

\theoremstyle{remark}
\newtheorem{remark}{Remark}[section]

\def\RR{\mathbb{R}}
\def\P{\mathscr{P}}
\def\L{\mathcal{L}}
\def\A{\mathcal{A}}

\def\P{\mathcal{P}}
\def\V{\mathcal{V}}

\def\la{\langle}
\def\ra{\rangle}
\def\raa{\rightarrow}

\DeclareMathOperator{\inte}{int}

\DeclareMathOperator{\lin}{lin}
\DeclareMathOperator{\wsup}{WSup}
\DeclareMathOperator{\winf}{WInf}
\DeclareMathOperator{\wmax}{WMax}
\DeclareMathOperator{\wmin}{WMin}
\DeclareMathOperator{\cl}{cl}
\DeclareMathOperator{\bd}{bd}
\DeclareMathOperator{\co}{co}
\DeclareMathOperator{\cone}{cone}
\DeclareMathOperator{\aff}{aff}
\DeclareMathOperator{\dom}{dom}

\DeclareMathOperator{\epi}{epi}

\DeclareMathOperator{\exepi}{\mathfrak{E} {\rm{pi}}}

\def\nplus{\boxplus}

\allowdisplaybreaks

\begin{document}

\title{New Representations of Epigraphs of Conjugate Mappings and Lagrange, Fenchel-Lagrange Duality   for Vector Optimization Problems }

\author{
	\name{
		N. Dinh\textsuperscript{a}\textsuperscript{b}\thanks{CONTACT N. Dinh. Email:  ndinh@hcmiu.edu.vn}
		\and D.H. Long\textsuperscript{c}\textsuperscript{d}
			}
			\affil{
		\textsuperscript{a}Department of Mathematics, International University Vietnam National University-Ho Chi Minh City, Ho Chi Minh City, Vietnam,  e-mail: ndinh@hcmiu.edu.vn;  
		\textsuperscript{b}Vietnam National University, Ho Chi Minh City,  Vietnam;	
		\textsuperscript{c}VNUHCM - University of Science, District 5, Ho Chi Minh city, Vietnam;
		\textsuperscript{d} Tien Giang University, Tien Giang town, Vietnam; e-mail: danghailong@tgu.edu.vn 
		}
}

\maketitle

\vskip-1cm

\begin{abstract}   In this paper we concern the vector problem of the model: 
\begin{align*}
({\rm VP})\quad\qquad &\winf \{F(x): x\in C,\; G(x)\in -S\}.
\end{align*}
where  $X, Y, Z$ are locally convex Hausdorff topological vector spaces,   $F\colon X\rightarrow Y\cup\{+\infty_{Y}\}$ and $%
G\colon X\rightarrow Z\cup\{+\infty_{Z}\}$ are proper mappings, $C$ is a
nonempty convex subset of $X$, and $S$ is a non-empty closed, convex   cone in $Z$. Several new presentations of epigraphs of composite 
conjugate mappings associated to  (VP) are established under variant qualifying conditions. The significance  of these representations is twofold: Firstly, they play a key role in establish new kinds of vector Farkas lemmas which serve as  tools in the study of vector optimization problems; secondly, they pay the way  to define   Lagrange dual problem and two new kinds of Fenchel-Lagrange dual problems for the vector problem (VP). 
 Strong and stable strong duality results corresponding to these  three mentioned dual problems of (VP) are established with the help of new Farkas-type  results just obtained from the representations.  
It is shown  that in the special case where $Y = \mathbb{R}$, the Lagrange  and Fenchel-Lagrange dual problems   for (VP), go back to  Lagrange dual problem,  and Fenchel-Lagrange dual problems for scalar problems, and the resulting  duality results  cover,  and in some setting, extend the corresponding  ones for scalar problems in the literature. 

\end{abstract} 

\begin{amscode}
49N15;  90C25; 90C29; 90C31; 90C46; 90C48
\end{amscode}

\begin{keywords} Vector inequalities;    (stable) vector Farkas lemmas; qualification conditions;
 extended epigraphs of conjugate mappings; Lagrange and Fenchel-Lagrange duality for vector optimization problems.

  \end{keywords}

\section{Introduction}
\label{sect1}

 We consider a  vector optimization problem of the form:  
\begin{align*}
({\rm VP})\quad\qquad\qquad &\winf \{F(x): x\in C,\; G(x)\in -S\}.
\end{align*}
where  $X, Y, Z$ are locally convex Hausdorff topological vector spaces,   $F\colon X\rightarrow Y\cup\{+\infty_{Y}\}$ and $%
G\colon X\rightarrow Z\cup\{+\infty_{Z}\}$ are proper mappings, $C$ is a
nonempty subset of $X$, and $S$ is a non-empty closed, convex   cone in $Z$.
Let  $A:=C\cap G^{-1}(-S)$ and assume along this paper that $A\cap \dom F\neq\emptyset$.  {Throughout this paper, we often say that the  triple $(F; G, C)$ defined the (VP). }

In the special case where $Y = \RR$, the problem {(VP)} reduces to the 
scalar problem 
\begin{align*}
({\rm P})\quad\qquad\qquad \inf \{f(x): x\in C,\; G(x)\in -S\}, 
\end{align*}
where $f : X \rightarrow \overline{\RR} := \RR \cup \{\pm \infty\}$.

Several dual problems of the (VP) are defined in the literature (see \cite{B12,BGW097nw,BGW09,CDLP20,DGLL17,DL2017,DL-ACTA-2020,  GP14,  Jahn-vector, Khan-Tammer-Zali, LT07,    Tanino92, Z83}). 
However, many among the dual problems in these works when specified  to the special case where $Y = \RR$ result to only Lagrangian dual problem of (P):  
\[ ({\rm D}_1)\quad \sup_{\lambda\in S^+}  \hskip0.6em\inf_{x\in C} (f+\lambda G)(x),  \]   
and consequently, strong duality results also go back to  Lagrange strong duality for (P). In \cite[pages 338, 373]{BGW09}, the author introduced some kind of dual problems for set-valued   optimization problems that can reduce (when $Y = \mathbb{R}$)  to some kind of Fenchel dual problem for  (P).
However,  just a  few  dual problems for  (VP) in the literature that can result to Fenchel or   Fenchel-Lagrange dual problems for  (P)
 {(see  \cite{Bot2010,BGWMIA09,DNV-08,DVN-08} and  reference therein) as }
 \begin{align*}
({\rm D}_2)\quad & \sup_{\substack{x^\ast \in X^\ast\\\lambda\in S^+}} \hskip0.6em [-f^*(x^\ast)-(i_C+\lambda G)^*(-x^\ast)], \\
({\rm D}_3)\quad  & \sup_{\substack{x^\ast,u^\ast\in X^\ast\\\lambda\in S^+}}  \!\! [-f^*(x^\ast)-{i^\ast_C}(u^\ast)- (\lambda G)^*(-x^\ast-u^\ast)], 
\end{align*}
where $i_C$ is the indicator function of $C$.  The difficulty is that the representation of epigraphs of conjugate mappings (in the case of vector problems being multivalued functions) becomes 
very complicated and it is so difficult to get some similar forms as for the scalar problem (P).

The aim of this paper is to introduce a way  to overcome this difficulty to establish some new ways of representations of epigraphs of conjugate mappings
that pay a general way  to define  some new dual problems for (VP) (also called Lagrange and Fenchel-Lagrange dual problems for (VP)) which extend the dual problems 
(${\rm D}_2$) and (${\rm D}_3$)    to vector setting. In other words, such Lagrange and Fenchel-Lagrange dual problems for (VP) in the case where $Y = \mathbb{R}$ go back to the corresponding dual problems (${\rm D}_1$), (${\rm D}_2$),  and (${\rm D}_3$)  for (P).

To overcome the    difficulty mentioned above in getting representations of epigraphs of conjugate mappings, we need some new notions on the order between subsets of $Y$ generated by $K$, denoted by $ \preccurlyeq_K $, and operations of subsets in product spaces and some key results  (called basic lemmas) on these subsets. 
Concretely, we use the weak order in $Y$ generated by a closed and convex cones $K$. Such kind of order are used frequently in the literature such as           {\cite{BGW09,CDLP20,DGLL17,DL-ACTA-2020,DL2017,Tanino92}.}  We introduce the space of nonempty subsets of $Y$, namely, $\left(\mathcal{P}_p(Y)^\bullet, \preccurlyeq_K\right) $,  {the} structure $(\mathcal{P}_p(Y)^\infty, \preccurlyeq_K, {\uplus})$, and the notion of ``extended epigarph" of conjugate mappings and the sum (denoted by $\mathfrak{E}{\rm pi }  F^*  $ and  $  \nplus$, respectively) of these extended epigraphs (e.g.,    
$\mathfrak{E}{\rm pi }  F^*  \nplus   \mathfrak{E}{\rm pi }  G^* $). 
Some basic properties of the sum, called  lemmas  are established (see Section 3). These form an important  cornerstone  for the main results of this paper.   

The paper is organized as follows: Section 2 give some preliminaries and some first results, where we recall the conjugate of a vector-valued functions and some properties of its epigraph. 
  In Section 3 we introduce some new notions with some basic lemmas  which can be considered as  key tools for the results in the rest of the paper.  Here we firstly introduce   ordered Space  $\left(\mathcal{P}_p(Y)^\bullet, \preccurlyeq_K\right) $ of nonempty subsets of $Y$   and the Structure $(\mathcal{P}_p(Y)^\infty, \preccurlyeq_K, {\uplus})$.  
  Next, we introduce the notion \emph{extended epigraph} of conjugate of a mapping  $F$, denoted by $\mathfrak{E}{\rm pi}  F^*$, and  the $\boxplus$-sum  of these extended epigraphs of conjugate mappings with some basic properties. 
Two basic lemmas (Lemmas \ref{exepi} and \ref{thm:3.1aa}) are established at the end of this section which serve as cornerstones  for the study  in the paper. Section 4  introduces certain regularity conditions for (VP),  under which, several presentations of epigraphs of conjugate mappings are established which serve as important cornerstones for the rest of the paper. The significance  of these representations is twofold: First, they play a key role in establish new kinds of vector Farkas lemmas which are important  in the study of vector optimization problems; secondly, they pay the way  to define   Lagrange and new kinds of Fenchel-Lagrange dual problems for the vector problem (VP) in Section 6.    In Section 5,  several versions of vector Farkas results for general systems and convex systems are established.  Section 6  introduces Lagrange dual problem and two forms of Fenchel-Lagrange dual problems for vector optimization problem (VP). Strong and stable strong duality  results for (VP) corresponding to  these   dual problems  are established. It is worth noticing that   in the special case where $Y = \mathbb{R}$, these dual problems collapse to the dual problems $({\rm D}_1)$, 
$({\rm D}_2 )$, and $ ({\rm D}_2)$ mentioned above, and,  better still, the duality results just obtained, go back   (cover) and extend  many the Lagrange duality and Fenchel-Lagrange duality  results for scalar problems in the literature (see, e.g.,  \cite{Bot2010,BGWMIA09,DNV-08,DVN-08}).

 \section{Preliminaries, Notations and First Results }
\label{section2}

Let $X,Y,Z$ be locally  convex  Hausdorff  topological vector spaces  with their topological dual spaces denoted by $X^{\ast },Y^{\ast }$ and $Z^{\ast }$, respectively. The only topology we consider on
dual spaces is the weak*-topology.
For a set $U\subset X$, we denote by $\inte U,\ \cl U,\ \bd U,\ \co U,\ \lin U,\ \aff U$, {and} $\cone U$ the \emph{interior},  the \emph{closure}, the \emph{boundary}, the \emph{convex hull}, the \emph{linear hull}, the \emph{affine hull}, and the \emph{cone hull} of $U$, respectively. 
For each $x\in X$, $\mathcal{N}_X(x)$ denotes the collection of all neighborhoods of $x$ in $X$. 
Assume that $W$ is a topological subspace of $X$. For $A\subset W$, denote by $\inte_{W} A$ the interior of $A$ w.r.t. the topology induced in  $W$.  
Let $K$ be a proper,  closed, and convex cone in $Y$ with nonempty interior, i.e.,  $\inte K\neq \emptyset $.  
It is worth observing that 
\begin{equation}   \label{int-plus}   
K+\inte K=\inte K.  
\end{equation}  

 {\it Weak Ordering  Generated by  a Convex Cone: Weak Infima and Weak Suprema. }    
We define  a \emph{weak ordering} in $Y$ generated by $K$ as follows: for all $y_1, y_2 \in Y$, 
\begin{equation}
y_{1}<_{K}y_{2}\; \Longleftrightarrow \; y_{1}-y_{2}\in -\inte K. 
\label{eq_weak_ordering_strict}
\end{equation}
In $Y$ we sometimes   also consider an usual   \emph{ ordering} generated by the cone $K$, $\leqq_K$, which is defined by $y_1\leqq_Ky_2$ if and only if $y_1-y_2\in -K, $ for $y_1, y_2 \in Y$.


We enlarge $Y$ by attaching a \emph{greatest element} $+\infty _{Y}$ and a 
\emph{smallest element} $-\infty _{Y}$ w.r.t. $<_{K}$, which do not
belong to $Y$, and we denote $Y^{\bullet }:=Y\cup \{-\infty _{Y},+\infty
_{Y}\}$. We assume by convention that  $-\infty _{Y}<_{K}y<_{K}+\infty _{Y}$ for any $y\in Y
$ and, for  $M\subset Y$,   
\begin{equation} \label{2.4}
\begin{gathered}
-\!(\!+\!\infty _{Y}\!)=-\!\infty _{Y}, \qquad
-\!(\!-\!\infty _{Y})=+\!\infty _{Y}, \\
(\!+\!\infty _{Y}\!)\!+\!y=y\!+\!(\!+\!\infty _{Y}\!)=+\!\infty _{Y},\quad \forall y\!\in\! Y\!\cup\!
\{\!+\!\infty _{Y}\!\},  \\
(\!-\!\infty _{Y}\!)\!+\!y=y\!+\!(\!-\!\infty _{Y}\!)=-\!\infty _{Y},\quad \forall y\!\in\! Y\!\cup\!
\{\!-\!\infty _{Y}\!\} ,    \\ 
M\!+\!\{\!-\!\infty _{Y}\!\}=\{\!-\!\infty _{Y}\!\}\!+\!M=\{\!-\!\infty _{Y}\!\},\ \ 
 M\!+\!\{\!+\!\infty_{Y}\!\}=\{\!+\!\infty _{Y}\!\}\!+\!M=\{\!+\!\infty _{Y}\!\}. 
 \end{gathered}
\end{equation}
The sums $(-\infty _{Y})+(+\infty _{Y})$ and $(+\infty _{Y})+(-\infty _{Y})$
are not considered in this paper.


{The following notions are the key ones of   the paper.}

\begin{definition}{\rm (\cite[Definition 7.4.1]{BGW09}, \cite{Tanino92})} \label{def1} Let $M\subset Y^{\bullet }$. 

$\left(\rm a\right) $ An element $\bar{v}\in Y^{\bullet }$
is said to be a \emph{weakly infimal element} of $M$ if for all $v\in M$ we
have $v\not<_{K}\bar{v}$ and if for any $\tilde{v}\in Y^{\bullet }$ such
that $\bar{v}<_{K}\tilde{v}$, then there exists some $v\in M$ satisfying $%
v<_{K}\tilde{v}$. The set of all weakly infimal elements of $M$ is denoted
by $\winf M$ and is called the \emph{weak infimum} of $M$.

$\left(\rm b\right) $ An element $\bar{v}\in Y^{\bullet }$ is said
to be a \emph{weakly supremal element} of $M$ if for all $v\in M$ we have $%
\bar{v}\not<_{K}v$ and if for any $\tilde{v}\in Y^{\bullet }$ such that $%
\tilde{v}<_{K}\bar{v}$, then there exists some $v\in M$ satisfying $\tilde{v}%
<_{K}v$. The set of all supremal elements of $M$ is denoted by $\wsup M$ and
is called the \emph{weak supremum} of $M$.

$\left(\rm c\right) $ The \emph{weak minimum} of $M$ is the set $%
\wmin M=M\cap \winf M$ and its elements are the \emph{weakly minimal elements%
} of $M$.

$\left(\rm d\right) $   The \emph{weak maximum} of $M$ is the set $\wmax M=M\cap \wsup M$ and its elements are the \emph{weakly maximal elements} of $M$.
\end{definition}


\begin{proposition}
\label{pro_decomp}
Let $\emptyset \ne M\subset Y^\bullet$. One has:
\begin{itemize}
\item[$(\rm i)$] $\wsup M\ne \{+\infty_Y\}$ {if and only if $ Y\setminus  (M-\inte K)  \ne \emptyset$.}

\item[$(\rm ii)$] For all $ y\in Y$, $\wsup (y+M)=y+\wsup M$. 
\end{itemize}
Assume further that $M\subset Y$ and  $\wsup M\ne \{+\infty_Y\}$ then it holds:
\begin{itemize}
\item[$(\rm iii)$] $\wsup M-\inte K=M-\inte K$.

\item[$(\rm iv)$]  The following decomposition\footnote{Here, by the term ``decomposition" we mean the sets  in the right-hand side of the equality are disjoint.}  of $Y$ holds 
  $$Y=(M-\inte K)\cup \wsup M\cup (\wsup M + \inte K).$$ 

\item[$(\rm v)$] $\wsup M=\cl(M-\inte K)\setminus (M-\inte K)$.


\item[$\rm(vi)$]   $\wsup (\wsup M+\wsup N)= \wsup  (M+\wsup N)=\wsup (M+N).$ 

\item[$(\rm vii)$] If $0_Y\in N\subset -K$ then $\wsup (M+N)=\wsup M$. \\
In particular, one has $\wsup (M-K)=\wsup (M-\bd K)=\wsup M$. 
\end{itemize}
\end{proposition}

\begin{proof} The assertions  (i) - (v) are quoted  from  \cite{DGLL17} and \cite{DL2017}     while    
  (vi)  from \cite[Proposition 7.4.3]{BGW09} (see also  \cite[Proposition 2.6]{Tanino92}). 

(vii)  If $0_Y\in N\subset -K$ then  it is easy to check that $M+N-\inte K=M-\inte K$. Taking (v) into account, we get 
$\wsup (M+N)=\wsup M$.  Moreover, as    $0_Y\in -\bd K\subset -K$, taking $N=  -\bd K$ and $N= -K$  successively in the previous equality, one gets  $\wsup (M-K)=\wsup (M-\bd K)=\wsup M$.  
\end{proof}

\begin{proposition}
\label{pro_nwwew}
Let $\emptyset \ne M\subset Y^\bullet$.  Then, only one of three following cases is possible: $\wsup M= \{+\infty_Y\}$, $\wsup M= \{-\infty_Y\}$, and $\emptyset\ne\wsup M\subset Y$.
\end{proposition}

\begin{proof}
Assume that $\wsup M\ne \{+\infty_Y\}$ and $\wsup M\ne \{-\infty_Y\}$.
It follows from  Definition \ref{def1} (b) that if $M\ni +\infty_Y$ then $\wsup M=\{+\infty_Y\}$  and that  $\wsup \{-\infty_Y\}=\{-\infty_Y\}$.  So, we can assume that $M\not\ni +\infty_Y$ and $M\ne\{-\infty_Y\}$, which yields $\emptyset\ne M\setminus\{-\infty_Y\}\subset Y$ (recall that $M\ne\emptyset$). 
On the other hand, it follows again from Definition \ref{def1} (b) that $\wsup M=\wsup M\setminus\{-\infty_Y\}$. So, replacing $M$ by $M\setminus\{-\infty_Y\}$ if necessary, we can assume that $\emptyset \ne M\subset Y$. 

Apply Proposition \ref{pro_decomp}(iii) and (v) (recall that $\wsup M\ne \{+\infty_Y\}$), one gets $\wsup M\ne\emptyset$ and   $\wsup M\subset Y$.
\end{proof}

\begin{remark}
\label{rem_1hh}
It is worth noting that $\winf M=-\wsup(-M)$ for all $M\subset Y^\bullet$. So, Propositions \ref{pro_decomp} and \ref{pro_nwwew}  hold true when  $\wsup$, $+\infty_Y$, $K$, and $\inte K$ are replaced by $\winf$, $-\infty_Y$, $-K$, and $-\inte K$, respectively.
\end{remark}


 {\it  Mappings and  Cone of Positive Operators.} 
Let $F \colon X\rightarrow Y^{\bullet }$ be a mapping. The \emph{domain},   $ \dom F$,  and the \emph{$K$-epigraph} of $F $, $\epi_{K}F $,   are defined  respectively by 
\begin{align*}
\dom F &:=\{x\in X:F (x)\neq +\infty _{Y}\}, \ \ 
\epi_{K}F: =\{(x,y)\in X\times Y:  F (x) \leqq_K  y  \}.
\end{align*}
 $F $ is  said to be \emph{proper} if $\dom F \neq \emptyset $ and $-\infty
_{Y}\notin F (X)$. {It is said to be \emph{$K$-convex} (resp., \emph{$K$-epi closed}) if $\epi_K F$ is a convex subset of $X \times Y$ (resp., $\epi_K F$ is a closed subset of the product space $X\times Y$, \cite{Bot2010}, \cite[Definition 5.1]{Luc}).}
  
The concept $K$-epi closed extends the concept lower semicontinuous (lsc, briefly) of a real-valued function.  The mapping  $F$ is said to be  \emph{positively $K$-lsc}  if $y^*\circ F$ is lsc for all $y^*\in K^+$ (see \cite{Bol98}, \cite[Definition 2.16]{Bol01})\footnote{This notion was used   in    \cite{Bot2010} and  \cite{JSDL05}  as  ``star $K$-lower semicontinuous''.}. According  to  \cite[Theorem 5.9]{Luc}, every positively $K$-lsc mapping is $K$-epi closed but the converse is not true. Moreover,   when $Y=\RR$, three notions lsc, positively $\RR_+$-lsc, and $\RR_+$-epi closed coincide with each other.

Denote by $\mathcal{L}(X,Y)$ the space of all
 continuous linear mappings  from $X$ to $Y$ equipped with   the   \emph{topology of point-wise convergence}, i.e.,   if  $(L_{i})_{i\in I}\subset \mathcal{L}(X,Y)$ and $L\in \mathcal{L}(X,Y)$, $L_{i}\rightarrow L$ in $\mathcal{L}(X,Y) $
means  $L_{i}(x)\rightarrow L(x)$ in $Y$ for all $x\in X$. The zero element of $\mathcal{L}(X,Y)$ is  $0_{\mathcal{L}}$.

Now let $S$ be a non-empty convex cone in $Z$. Recall that the {\it cone of positive operators} (see \cite{AB85,KLS89}) 
is defined 
by
$\L_+(S,K):=\{T\in\L(Z,Y): T(S)\subset K\}$. 
When  $Y = \mathbb{R}$, this  cone collapse to the positive dual cone $S^+:=\{ z^* \in Z^\ast :\la z^*, s \ra \geq 0, \forall s \in S \}$ of $S$. 
\medskip

 {\it  Conjugate Mappings  of Vector-Valued Functions.  } The following notion of conjugate mapping of a mapping is specified from the corresponding  one for set-valued mappings in   \cite[Definition 7.4.2]{BGW09} and \cite[Definition 3.1]{Tanino92}.

\begin{definition}  \label{def2.3} For  $F\colon X\rightarrow Y\cup\{+\infty_{Y}\}$,  the set-valued mapping $F
^{\ast }\colon \mathcal{L}(X,Y)\rightrightarrows Y^{\bullet }$ defined by 
$F ^{\ast }(L):=\wsup\{L(x)-F (x):x\in X\}$
is called the \emph{conjugate mapping} of $F $. The \emph{$K$-epigraph}  and the \emph{domain} of $F^*$  are, respectively,
\begin{align*}
\epi_{K}\!\! F ^{\ast }&:=\big\{(L,y)\in \mathcal{L}(X,Y)\times Y :   y \in     F
^{\ast }(L) + K \!\big\}, \\
 \dom\! F ^{\ast }&:=\big\{L\in \mathcal{L}(X,Y)  : F ^{\ast }(L)\!\neq\!
\! \{+\infty
_{Y}\}\!\big\}. 
\end{align*}
\end{definition}
It is well-known that   if $F\colon X\rightarrow Y\cup\{+\infty_{Y}\}$ is a proper
mapping then $\epi F^{\ast}$ is a closed subset of $
\mathcal{L}(X,Y)\times Y$ { (see  \cite[Lemma {3.5}]{DGLL17}).} Moreover, the next characterization of $ \epi_K F ^{\ast } $ (see \cite[Theorem 3.1]{DGLMJOTA16}) is of importance in the sequent:
\begin{eqnarray} \label{epiF-star}
(L,y)\in \epi_K F ^{\ast }   &\Longleftrightarrow& [F (x)-L(x)+y\notin - \inte K,\; \forall x\in X]. 
\end{eqnarray} 
For a subset $D\subset X$, the \emph{indicator map} $
I^X_D\colon X\rightarrow {Y}^{\bullet }$ is defined by $I^X_D (x) = 0_Y $ if $x \in D$ and $I^X_D(x) = +\infty_Y $, otherwise.  
As the cone $K$ in $Y$  is fixed throughout the  paper, for the sake of simplicity,  from now on we will write $\epi F$ and $\epi F^*$
instead of $\epi_KF$ and $\epi_K F^*$, respectively. By the same reason we will write $I_D$ instead of $I^X_D$ 
where there is no confusion.

\section{New Tools and Basic Lemmas} 
\label{section3}

{In this section we firstly introduce the order space $\left(\mathcal{P}_p(Y)^\bullet, \preccurlyeq_K\right) $ and the structure $(\mathcal{P}_p(Y)^\infty, \preccurlyeq_K, {\uplus})$ with the so-called {\it ``WS-sum".} New notions  of extended epigraphs of conjugate mappings  and  a new operator on these sets, called the  $\boxplus$-sum.
Next, we prove { two basic lemmas   (serves as  key tools of our results). One among them (Lemma \ref{thm:3.1aa})   establishes } the relation between the ``normal" epigraph (defined in Section 2) of the conjugate of a sum of two mappings and  the   $\boxplus$-sum  of  two extended epigraphs   of the conjugates of  such two mappings.         }

\subsection{The Ordered Space  $\left(\mathcal{P}_p(Y)^\bullet, \preccurlyeq_K\right) $  and The Structure $(\mathcal{P}_p(Y)^\infty, \preccurlyeq_K, {\uplus})$}

{ \it  The Ordered Space  $\left(\mathcal{P}_p(Y)^\bullet, \preccurlyeq_K\right) $.}  
Let $\mathcal{P}_0(Y^\bullet)$ be the collection  of all non-empty subsets of $Y^\bullet$.  
The  ordering ``$\preccurlyeq_K$''  on   $\mathcal{P}_0(Y^\bullet)$  is defined   \cite{DL2017}   as, 
 for $M,N \in \mathcal{P}_0(Y^\bullet)$, 
\begin{equation}
\label{eq_6.33a}
M \ \preccurlyeq_K N \ \     \Longleftrightarrow \ \   \left( v\not<_Ku, \; \forall u\in M,\; \forall v\in N\right).  
\end{equation}
For other orderings on $\mathcal{P}_0(Y^\bullet)$, see, e.g.,  \cite{Kuroiwa98}.

\begin{proposition}\label{prop_4gg} The following assertions hold 
\begin{itemize}
\item[$\rm(i)$]  For all  $ \emptyset \ne M,N\subset Y$, one has:  
\begin{eqnarray*}  M  \preccurlyeq_K N     \ \     &\Longleftrightarrow& \ \      N\cap (M-\inte K)= \emptyset  \\ \ \ 
    &\Longleftrightarrow& \ \  M\cap (N+\inte K)= \emptyset.
    \end{eqnarray*}
\item[$\rm(ii)$] For all $M\in \mathcal{P}_0(Y^\bullet)$, one has  $\winf M \preccurlyeq_K M \textrm{ and } M  \preccurlyeq_K \wsup M. $

\item[$\rm(iii)$] $\textrm{If } M\subset N \subset Y^\bullet$ then $\wsup M \preccurlyeq_K \wsup N$.

\end{itemize}
\end{proposition}

\begin{proof}
 $\rm(i)$ and  $ \rm(ii)$ follow easily from  \eqref{eq_6.33a}  while  (iii)   is   \cite[Proposition 1]{DL2017}.
\end{proof}

A subset $U    \subset   Y  $  is called a {\it $(Y,K)$-partition style subset of $Y$}  if  the following decomposition of $Y$  holds 
\begin{equation} \label{defPp}
Y=(U-\inte K)\cup U\cup (U+\inte K).
\end{equation}  
Denote   by $\mathcal{P}_p(Y)$ the collection of  all  $(Y,K)$-partition style subsets of $Y$ and  set 
$$\mathcal{P}_p(Y)^\bullet := \mathcal{P}_p(Y) \cup \{ \{ + \infty_Y\}, \{ - \infty_Y\}\}.$$ 
It is obvious  that  if  $M\subset Y^\bullet$ then  $\pm\wsup M \in \mathcal{P}_p( Y)^\bullet$,  $\pm\winf M \in \mathcal{P}_p( Y)^\bullet$ and (by \eqref{eq_6.33a}),   for any  $U \in \mathcal{P}_p(Y)$,  one has $U   \preccurlyeq_K \{ + \infty_Y\}$ and $\{ - \infty_Y\}  \preccurlyeq_K U$.    
Moreover,   $\left(\mathcal{P}_p(Y)^\bullet, \preccurlyeq_K\right) $  is an ordered space \cite{DL2017}, i.e.,  $\preccurlyeq_K$ is a partial order on $\mathcal{P}_p(Y)^\bullet $ with properties: reflexive, anti-symmetric,
and transitive.


\begin{proposition}\label{pro_5hh}  Let  $U,V\in \mathcal{P}_p(Y)^\bullet$.     Then 
\begin{itemize}
\item[$\rm(i)$]    If\  $U\subset V$ then $U=V$, 

\item[$\rm(ii)$]     If\  $U\in \mathcal{P}_p(Y)$ then 
\begin{align}
U+K=U\cup (U+\inte K),  \ \    
U-K=U\cup (U-\inte K).     \label{eq_7gg}
\end{align}
\end{itemize}
\end{proposition}

\begin{proof}
$\rm(i)$ is   \cite[Proposition 4]{DL2017}. For the proof of 
$\rm(ii)$,   take $U\in \mathcal{P}_p(Y)$.    It is clear that  
\begin{equation} \label{1a} 
U+K\supset U\cup (U+\inte K). \end{equation} 
Next, we  claim that $(U+K)\cap (U-\inte K)=\emptyset$. Indeed, assume the contrary that  $(U+K)\cap (U-\inte K)\ne\emptyset$. Then, there exist $u,v\in U$, $k\in K$, and $k_0\in \inte K$ such that $u+k=v-k_0$, and hence,  $u=v-k_0-k\in U\cap (U-\inte K)$, which contradicts  the decomposition   \eqref{defPp}.
  So,  $(U+K)\cap (U-\inte K)=\emptyset$, and 
$U+K\subset U\cup (U+\inte K). $
The first equality in \eqref{eq_7gg}  now follows from \eqref{1a} while that  of the second one   is similar.  
\end{proof}

 {\it The structure $(\mathcal{P}_p(Y)^\infty, \preccurlyeq_K, {\uplus})$.}  We now introduce a  new kind of  ``sum" of two sets (called ``WS-sum")  on the collection $\mathcal{P}_p(Y)^\infty :=    \mathcal{P}_p(Y)\cup \{\{+ \infty_Y\}\} $, which will be one of the main tools for our main results in the next sections. 
  
  \begin{definition}\label{def_3.1zz}
 For  $U,V\in  \mathcal{P}_p (Y)^\infty  $, the  {\it WS-sum} of $U$ and $V$, denoted by  $U\uplus V$,  is  a set from  $\mathcal{P}_p (Y)^\infty $  and is  defined by
\begin{equation}\label{newsum}
 U\uplus V\   :=   \  \wsup (U+V).
\end{equation}
\end{definition}

\begin{lemma}\label{pro_6hh}
Let  $U,V\in  \mathcal{P}_p(Y)$. Then
\begin{itemize}     
\item[$\rm(i)$] the following decompositions of $Y$  hold:
$$Y=(U-\inte K)\cup (U+K)=(U-K)\cup (U+\inte K), $$
\item[$\rm (ii)$]   it holds: $U\preccurlyeq_K V \Longleftrightarrow V\subset U+K $

\hskip2.6cm$\Longleftrightarrow U\subset V-K$, 

\medskip

\item[$\rm(iii)$]   $\wsup U=\winf U=U$.   
\end{itemize}
\end{lemma}

\begin{proof}
It is easy to see that (i)   is a direct consequence of \eqref{defPp} and Proposition \ref{pro_5hh}(ii) while  $\rm (ii)$ follows  from  (i) and Proposition \ref{prop_4gg}(i).   
 For (iii), as  $U\in  \mathcal{P}_p(Y)$,  \eqref{defPp} holds
yielding  $Y\setminus (U-\inte K) \neq \emptyset$, and hence, by   Proposition \ref{pro_decomp}(i),   $\wsup U\ne\{+\infty_Y\}$. Now, according to  Proposition \ref{pro_decomp}(v), $\wsup U=\cl(U-\inte K)\setminus (U-\inte K)$.

On the other hand, as $U+\inte K$ is  open,  $U-K$ is closed (see (i)). 
We  now show that $U-K$ is the   smallest closed subset of $Y$ containing $U-\inte K$.  Indeed, assume that $M$ is a closed subset containing $U-\inte K$, we will prove that $M\supset U-K$. Take $w\in U-K$. Then, there are $u\in U$ and $k\in K$ such that $w=u-k$. Pick $k_0\in \inte K$. It is easy to see that $w_n:=u-k-\frac{1}{n}k_0\in U-K-\inte K=U-\inte K\subset M$ for all $n\in \mathbb{N}^*$ and $w_n\to w$. So, by the closedness of $M$, $w\in M$ and 
 $\cl(U-\inte K)=U-K$.  Thus, by Proposition \ref{pro_decomp}(v),
 $$\wsup U=\cl(U-\inte K)\setminus (U-\inte K)=(U-K)\setminus (U-\inte K).$$ 
  This, together with \eqref{eq_7gg},   yields $\wsup U=U$ (note that $U\cap (U-\inte K)= \emptyset$).
The proof of the equality $\winf U=U$ is  similar. \end{proof}

We now give some  properties of the structure $(\mathcal{P}_p(Y)^\infty, \preccurlyeq_K, {\uplus})$.   

\begin{proposition} \label{prop_1ab}  Let  $U,V,W, U',V'\in \mathcal{P}_p(Y)^\infty$, $y\in Y$.  One has
\begin{itemize}
\item[$\rm(i)$] $U\uplus (-\bd K)=U$,

\item[$\rm(ii)$]  $U\uplus V=V\uplus U$   (commutative), 

\item[$\rm(iii)$]  $(U\uplus V)\uplus W= U\uplus (V\uplus W)$ (associative),

\item[$\rm(iv)$]  If $U\preccurlyeq_K V$ then $U\uplus W\preccurlyeq_K V\uplus W$  (compatible of  the sum $\uplus$ with $\preccurlyeq_K$), 

\item[$\rm(v)$] If $U\preccurlyeq_K V$ and $U'\preccurlyeq_K V'$ then $U\uplus U'\preccurlyeq_K V\uplus V'.$

\item[$\rm(vi)$]   $y\in (U\uplus V)+K$ if and only if there exists $W'\in \mathcal{P}_p(Y)$ such that  $U \preccurlyeq_K W'$ and $y\in W'\uplus V$.
\end{itemize}
\end{proposition}
\begin{proof} 
(i) and (ii) follow easily  from Definition \ref{def_3.1zz} (see Proposition \ref{pro_decomp} (vii)). For 
(iii),  one has, by Proposition \ref{pro_decomp}(vi), 
\begin{eqnarray*} (U\uplus V)\uplus W&=& \wsup (\wsup (U+V)+W)=\wsup (U+V+W)  = U\uplus (V\uplus W). 
\end{eqnarray*} 
\indent (iv)  holds trivially if  $V\uplus W:=\wsup(V+W)=\{+\infty_Y\}$. 
 Assume now that $\wsup (V+W)\subset Y$.   Consequently, $V,W\subset Y$.   As $U\preccurlyeq_K V$, one has $U\subset Y$ and
\begin{align*}
&U\subset V-K \quad \textrm{(by Lemma \ref{pro_6hh}(ii))}\notag\\
\Longrightarrow\; & U+W\subset V+W-K\notag\\
\Longrightarrow\; &\wsup (U+W) \preccurlyeq_K  \wsup (V+W-K) \quad \textrm{(by Proposition \ref{prop_4gg}(iii)).}\label{eq_10nwnw}
\end{align*}
One also has, by  Proposition \ref{pro_decomp}(vii), $\wsup (U+V-K)=\wsup (U+V)$. So, $\wsup (U+W) \preccurlyeq_K  \wsup (V+W) $.  In other words, $U\uplus W\preccurlyeq_K  V\uplus W$.  

 (v) It follows from (iv) and  the transitive property of $\preccurlyeq_K$.


(vi)  
If  $y\in (U\uplus V)+K$  then there is $k\in K$ such that  (see Proposition \ref{pro_decomp}(ii))   
$$y\in (U\uplus V)+k=\wsup (U+V)+k=\wsup (U+V+k) =(k+U)\uplus V. $$ 
 So, if we take $W'=k+U$ then   $W'\in \mathcal{P}_p(Y)$,   $y\in W'\uplus V$,   and $W'\subset U+K$, meaning that  $U \preccurlyeq_K W'$  by Lemma  \ref{pro_6hh}(ii).    

Conversely, if  there exists $W'\in \mathcal{P}_p(Y)$ such that  $U \preccurlyeq_K W'$  then   
$U\uplus V \preccurlyeq_K W'\uplus V$ 
 (by (iv)). So, if  further,  $y\in W'\uplus V$ then it means that $W'\uplus V\ne \{+\infty\}$, and hence, $W'\uplus V \subset (U\uplus V)+K$  (again, by  Lemma \ref{pro_6hh}(ii)), yielding  $y\in  (U\uplus V)+K$. 
\end{proof}







\subsection{
Extended Epigraphs of Conjugate Mappings and Their $\nplus$-Sums}

\begin{definition} \label{def_exepi} $\phantom{x}$
(a)	The  {\it $K$-extended epigraph} of the conjugate mapping $F^*$ is 
\begin{equation} \label{extepi}
\mathfrak{E}{\rm pi}  F^*:=\{(L,U)\in \L(X,Y)\times \mathcal{P}_p(Y): L\in \dom F^*,\;  F^*(L)\preccurlyeq_K U\}. 
\end{equation} 
	
(b) For $\mathscr{M}_1,\mathscr{M}_2\subset  \L(X,Y)\times \mathcal{P}_p(Y)$,   the {\it $\nplus$-sum}  of these two sets is defined as: 
\begin{equation}\label{eq_12abc}
\mathscr{M}_1\nplus\mathscr{M}_2:=\{(L_1+L_2, U_1\uplus U_2):(L_1,U_1)\in \mathscr{M}_1,\; (L_2,U_2)\in \mathscr{M}_2\}.
\end{equation}
\end{definition}
In particular, if  $F, G : X \raa Y^\bullet$ then the $\nplus$-sum of  $\mathfrak{E}{\rm pi }  F^* $ and  $   \mathfrak{E}{\rm pi }  G^*$  is
$$
\mathfrak{E}{\rm pi }  F^*  \nplus   \mathfrak{E}{\rm pi }  G^*= \{(L_1+L_2, U_1\uplus U_2)\, : \, (L_1,U_1)\in \mathfrak{E}{\rm pi}  F^\ast,\; (L_2,U_2)\in \mathfrak{E}{\rm pi}  G^*   \}. 
$$
We can understand simply that   the extended  epigraph of $F^*$, $\mathfrak{E}{\rm pi}  F^*$,   is   the ``epigraph" of   $F^\ast$ which is  considered as a single valued-mapping  $F^\ast \colon \L(X, Y) \raa  ( \mathcal{P}_p(Y),\preccurlyeq_K)$ and the ``epigraph" here will be understood  in the same way as the one of a real-valued function.  

It is also worth observing that from  the definition of $\nplus $-sum  and  Proposition \ref{prop_1ab},   the     $\nplus$-sum  is  commutative and associative on     $\L(X,Y)\times \mathcal{P}_p(Y)$.  

Note  that $\exepi F^\ast \subset   \L(X,Y)\times \mathcal{P}_p(Y)$ while    $\epi F^\ast \subset   \L(X,Y)\times Y $.  We  define  the set-valued mapping $\Psi$ as follows: 
\begin{align} \label{Psi}
\Psi \colon \L(X,Y)\times \mathcal{P}_p(Y)&\rightrightarrows \L(X,Y)\times Y     \nonumber   \\
(L,U)\ \ \ \ \ \ &\mapsto \Psi (L,U):=\{L\}\times U.  
\end{align}
It is easy to verify  that for  $\mathscr{M}, \mathscr{N}, \mathscr{Q},  \mathscr{M}_i  \subset   \L(X,Y)\times \mathcal{P}_p(Y) $,  for all $i \in I$  ($I$ is an arbitrary index set), it holds 
 \begin{eqnarray} 
 &&\mathscr{M} \subset \mathscr{N}   \ \ \ \Longrightarrow \ \ \    \Psi(\mathscr{M}) \subset \Psi (\mathscr{N}),   \label{eqbbb}    \\
 &&\mathscr{M} \subset \mathscr{N}   \ \ \ \Longrightarrow \ \ \      \mathscr{M}  \nplus \mathscr{Q}      \subset \mathscr{N}   \nplus \mathscr{Q},  \    \label{eqccc} \\
&&\bigcup\limits_{i \in I}  \left(\mathscr{M}_i \boxplus\mathcal{N}\right)=\left(\bigcup\limits_{i \in I}  \mathscr{M}_i \right) \boxplus\mathcal{N},  \ \ \text{and} \label{eqplusunion}\\
 &&\Psi \big(\bigcup\limits_{i \in I}  \mathscr{M}_i \big)\ \  = \ \  \bigcup\limits_{i \in I} \Psi\big(\mathscr{M}_i\big). \label{Psi-cup} 
    \end{eqnarray} 
   The 
 relation between   $ \exepi F^\ast$ and   $\epi F^\ast$  is  given in the next proposition.

\begin{proposition}\label{rel_epi} 
Let $F\colon X\to Y^\bullet$ be  a proper mapping.  Then $\Psi (\exepi F^*) = \epi F^*$.
\end{proposition}

\begin{proof}
 Assume that  $(L,y)\in \Psi (\exepi F^*)$. Then, there exists 
$U\in \mathcal{P}_p(Y)$ such that $(L,U)\in \exepi F^*$ and $y\in U$. As $(L,U)\in \exepi F^*$, one has $F^*(L)\preccurlyeq_K U$, or equivalently, $U\subset F^*(L)+K$ (by {Lemma} \ref{pro_6hh}(ii)). So, $y\in U\subset F^*(L)+K$, and hence, $(L,y)\in \epi F^*$.

 Assume now that $(L,y)\in \epi F^*$. Then   $y\in F^*(L)+K$. 
As  $\winf K=\bd K \subset K$,  $y+\winf K \subset F^*(L)+K$ which yields  $(L,y+\winf K)\in \exepi F^*$. It is clear that $y\in y+\winf K$ and $y+\winf K\in \mathcal{P}_p(Y)$. So, $(L,y)\in  \Psi(\exepi F^*)$ and we are done.
\end{proof}

\begin{remark}
\label{rem_2dd}
Coming back  to the scalar case,   when $Y=\mathbb{R}$ and $K=\mathbb{R}_+$,  one has $\mathcal{P}_p(Y)={\mathbb{R}}$ while the order ``$\preccurlyeq_K$'' (see  \eqref{eq_6.33a})  and the sum ``$\uplus$''    (see \eqref{newsum}) become the normal order ``$\le$'' and the usual  sum ``$+$''  on the set of extended real numbers, respectively.
Hence, $\L(X,Y)\times \mathcal{P}_p(Y)=X^*\times {\mathbb{R}}$ and the $\boxplus$-sum defined in  \eqref{eq_12abc} collapses to  the usual Minkowski sum of two subsets in $X^*\times {\mathbb{R}}$. In this special case, the mapping $F$ becomes an extended-real-valued function, and the conjugate mapping  $F^\ast$ collapses to the usual conjugate function in the sense of convex analysis. Consequently,  
both $K$-epigraph and $K$-extended epigraph of  conjugate mappings  collapse to their  usual  epigraphs in the sense of   convex analysis. The mapping $\Psi$ (defined by \eqref{Psi})  in this case   is nothing else but the identical mapping of $X^*\times {\mathbb{R}}$. In  other words, if $\mathscr{M}_1, \mathscr{M}_2
\subset  X^*\times {\RR}$, one has
 $\Psi(\mathscr{M}_1 \nplus \mathscr{M}_2)=  \Psi (\mathscr{M}_1 + \mathscr{M}_2) = \mathscr{M}_1 + \mathscr{M}_2.$
\end{remark}


\subsection{Basic Lemmas }


Let $F_1,F_2\colon X\to Y^\bullet$ be proper $K$-convex mappings. 
We say that the regularity condition $(C_0)$ holds for  $F_1$ and $F_2$ (in this order)  if there holds: 

\medskip

\begin{tabular}{c  c}
$(C_0)$ &  
\begin{minipage}{0.9\textwidth}
$\exists\, \widehat x\in  \dom F_1:  F_2\textrm{ is continuous  at }\widehat x$.
\end{minipage}
\end{tabular}

\begin{lemma}[{Basic lemma 1}]
\label{exepi}
Let $F_1,F_2\colon X\to Y^\bullet$ be proper $K$-convex mappings, $M\subset Y$ be nonempty. Assume that  the condition  $(C_0)$ holds for $F_1$ and $F_2$.  
If   
$$\bar y\in Y\setminus [M+(\bar L-F_1-F_2)(X)-\inte K]$$ 
for some   $\bar L\in \L(X,Y)$, 
then  there exist  $L_1,L_2\in\L(X,Y)$ such that $L_1+L_2=\bar L$, $\wsup [M+(L_1-F_1)(X) +(L_2-F_2)(X)]\ne \{+\infty\}$  and
\begin{equation}
\label{eq:15neww}
\bar y\notin M+(L_1-F_1)(X) +(L_2-F_2)(X)-\inte K.
\end{equation}
\end{lemma}

\begin{proof} (See Appendix A). 
\end{proof}

\begin{lemma}
\label{pro:3.3_nwww}
Let $F_1,F_2,F_3\colon X\to Y^\bullet$ be proper mappings. Then, it holds:
\begin{itemize}
\item[$\rm(i)$] $\exepi F_1^\ast \boxplus \exepi F_2^\ast\subset \exepi (F_1+F_2)^\ast$,

\item[$\rm(ii)$] $\Psi(\exepi F_1^\ast \boxplus \exepi F_2^\ast)\subset \epi (F_1+F_2)^\ast$,

\item[$\rm(iii)$]  $\Psi\left(\exepi F_1^\ast\boxplus\exepi F_2^\ast\boxplus \exepi F_3^\ast \right)\subset  \Psi\left(\exepi (F_1+F_2)^\ast\boxplus \exepi F_3^\ast\right).$
\end{itemize}
\end{lemma}

\begin{proof} (see Appendix B). 
\end{proof}

\begin{lemma}[{Basic lemma 2}]
\label{thm:3.1aa}
Let $F_1,F_2, F_3\colon X\to Y^\bullet$ be proper $K$-convex mappings and assume that the  condition $(C_0)$ holds for $F_1$ and $F_2$.  Then, one has:
\begin{itemize}
\item[$\rm(i)$] $\epi (F_1+F_2)^*=\Psi(\exepi F_1^\ast \boxplus \exepi F_2^\ast)$,

\item[$\rm(ii)$] $\Psi\left(\exepi (F_1+F_2)^\ast\boxplus \exepi F_3^\ast\right)=\Psi\left(\exepi F_1^\ast\boxplus \exepi F_2^\ast\boxplus\exepi F_3^\ast \right).$

\item[$\rm(iii)$] If, in addition that   one of the following conditions holds\\
\begin{tabular}{c  c}
$(C_0')$ &  
\begin{minipage}{0.9\textwidth}
$\exists \tilde x\in  \dom F_1\cap \dom F_2:  F_3\textrm{ is continuous  at } \tilde x$,
\end{minipage}
\end{tabular}\\
\begin{tabular}{c  c}
$(C_0'')$ &  
\begin{minipage}{0.9\textwidth}
$\exists \bar x\in  \dom F_3:  F_1 \textrm{ and } F_2\textrm{ are continuous  at } \bar  x$, 
\end{minipage}
\end{tabular}
\end{itemize}
\noindent then 
 \begin{eqnarray}\label{eqthm31}
 \epi (F_1+F_2+F_3)^\ast   &=&   \Psi\left(\exepi (F_1+F_2)^\ast\boxplus \exepi F_3^\ast\right)    \nonumber \\
     &=& \Psi\left(\exepi F_1^\ast\boxplus \exepi F_2^\ast\boxplus\exepi F_3^\ast \right).
 \end{eqnarray} 
 \end{lemma}

\begin{proof} (see Appendix C). 
\end{proof}


\section{Representations of the Epigraphs of Conjugate Mappings   }

Let $X,Y$, $Z$  be  locally convex topological vector spaces,   and  $K, S$ be  non-empty convex cones in $Y$ and $Z$, respectively, with  $\inte K\ne \emptyset$.   
{ Let  further that $(F;G,C)$ be the triple  defined the problem (VP) as in Section 1 with the assumption that $A\cap \dom F\neq\emptyset$, where    $A:=C\cap G^{-1}(-S)$ is the feasible set of (VP). }

{  In this section, we  establish the main results of the paper:  representations of  the epigraph of the conjugate of the mapping $F + I_A$, 
$\epi (F+I_A)^\ast$, in terms of  epigraphs of the conjugate mappings  of its members  $F$, $G$,  and $I_C$.}

Concerning  the triple $(F; G, C)$,  let us   set 
\begin{align*}
\mathfrak{A}_1&:= \bigcup_{T\in\L_+(S,K)}\exepi (F+I_C+T\circ G)^{\ast },\\
\mathfrak{A}_2&:= \exepi F^*\boxplus \bigcup_{T\in \L_+(S,K)}\exepi ({I_{C}}+T\circ G)^{\ast }, \\
\mathfrak{A}_3&:=\exepi F^* \boxplus\exepi {I_{C}^*}\boxplus \bigcup_{T\in\L_+(S,K)}\exepi( T\circ G)^*,
\end{align*}
{and consider the following  sets:}
\begin{equation} \label{eq21}
\A_i=\Psi (\mathfrak{A}_i),    \quad i=1,2,3,
\end{equation}
where $\Psi$ is the mapping defined in  \eqref{Psi}.  {It is worth observing that   the  sets $\A_1$  is  exactly the qualifying set 
  proposed recently in \cite{DGLL17,DGLMJOTA16} {(see Lemma \ref{rem:3b} below).} Moreover, in the special case where $Y = \RR$, and $K = \RR_+$, the sets $\A_i$, $i = 1, 2, 3$ go back to the well-known sets appeared in convex scalar optimization theory       (see Remark \ref{rem73}  below).}

{\begin{lemma}
\label{rem:3b}
It hold
\begin{gather}
 \epi (F+I_{A})^{\ast}  \ \ \   \supset\bigcup_{T \in \mathcal{L}_+(S, K)}
\epi(F+I_{C}+T\circ G)^{\ast} \ \ \ \textrm{and}      \label{eqLem41}   \\ 
 \A_1\  = \bigcup_{T\in\L_+(S,K)}\!\epi (F+I_C+T\circ G)^{\ast }. \label{eqLem41b} 
\end{gather}
\end{lemma}
}

{\begin{proof}  Note that \eqref{eqLem41} is  \cite[Lemma 4.1]{DGLMJOTA16}. For the proof of \eqref{eqLem41b}, 
observe firstly that  by the definitions of $\A_1$, $\mathfrak{A}_1, $   and by \eqref{Psi-cup},   one has 
\begin{align*}
\A_1&= \Psi\Big[\!\!\!\bigcup_{T\in\L_+(S,K)}\exepi (F+I_C+T\circ G)^{\ast }\Big]      = \!\!\! \bigcup_{T\in\L_+(S,K)}\!\!\!\Psi[\exepi (F+I_C+T\circ G)^{\ast }] \\
          &=  \bigcup_{T\in\L_+(S,K)}\!\epi (F+I_C+T\circ G)^{\ast }
\end{align*}
(the last equality follows from Proposition \ref{rel_epi}), which means that   \eqref{eqLem41b} is proved. 
 \end{proof}
}

\begin{proposition} \label{pro:9}
The next inclusions hold:  
$
\A_3 \subset  \A_2  \subset  \A_1  \subset \epi (F+I_A)^\ast . 
$
\end{proposition}

\begin{proof}
 {{It follows from Lemma \ref{pro:3.3_nwww}(i)  that  $\mathfrak{A}_3\subset \mathfrak{A}_2\subset \mathfrak{A}_1$, and taking   \eqref{eqbbb} into account, we get    $\A_3\subset \A_2\subset \A_1$. }}      On the other hand, by Lemma    \ref{rem:3b},   
\begin{equation} \label{eqprop41}
\A_1 \ \   =   \bigcup_{T \in \mathcal{L}_+(S, K)}
\epi (F+I_{C}+T\circ G)^{\ast}  \ \ \subset \ \  \epi (F+I_{A})^{\ast},  
\end{equation} 
and the proof is complete. 
\end{proof}


{In the rest of this section,  we assume  that {\it $F$ is $K$-convex,  $G$ is $S$-convex, and  that $C$ is a convex subset of $X$.}}
Consider the following regularity  conditions:

\medskip 

\noindent 
\begin{tabular}{c  c}
$(C_1)$ &  
\begin{minipage}{0.9\textwidth}
$\exists x_1 \in  C\cap \dom F: G(x_1)\in -\inte S$
\end{minipage}
\end{tabular}

\noindent
\begin{tabular}{c  c}
$(C_2)$ &  
\begin{minipage}{0.9\textwidth}
$\exists  x_2\in  C\cap\dom G:  F\textrm{ is continuous  at } x_2$.
\end{minipage}
\end{tabular}

\noindent
\begin{tabular}{c  c}
$(C_3)$ &  
\begin{minipage}{0.9\textwidth}
$\exists  x_3\in  C :  G\textrm{ is continuous  at } x_3$.
\end{minipage}
\end{tabular}


\begin{theorem}[Representation of $\epi (F+I_A)^\ast$] The next assertions hold: 
\label{thm:4.1f}\\
\indent {\bf (a)}  If $(C_1)$ holds then 
$\epi (F+I_A)^*=\A_1, $\\
\indent {\bf (b)} If $(C_1)$ and $(C_2)$ {hold} then 
$\epi (F+I_A)^\ast=\A_2, $\\
\indent {\bf (c)} If $(C_1)$, $(C_2)$, and $(C_3)$ {hold} then 
$\epi (F+I_A)^*=\A_3.$
\end{theorem}

\begin{proof} 
{{\bf (a)}  follows from Lemma \ref{rem:3b}   and    \cite[Theorems 4.3]{DGLL17}.  }

{\bf (b)} 
 Firstly, if  $(C_1)$ holds then by  {\bf (a)}, one has 
\begin{align}
\epi (F+I_A)^\ast&=\A_1= \bigcup_{T\in\L_+(S,K)}\epi (F+I_C+T\circ G)^{\ast }. \label{eq:33f}
\end{align}
{If further  that $(C_2)$ holds,  then $(C_0)$   holds with $F_1=I_C+T\circ G$ and $F_2=F$. It now follows from  Basic Lemma 2 (Lemma \ref{thm:3.1aa}(i))  that,   for each $T\in \L_+(S,K)$,   }
\begin{equation*}
\epi (F+I_C+T\circ G)^{\ast }=\Psi\big[\exepi F^\ast \boxplus \exepi(I_C+T\circ G)^{\ast }\big], 
\end{equation*}
{which, together with    \eqref{eq:33f}, gives us }
 \begin{eqnarray*} 
 \epi (F+I_A)^\ast&=& \bigcup_{T\in\L_+(S,K)}\epi (F+I_C+T\circ G)^{\ast }\\
 &=&  \bigcup_{T\in\L_+(S,K)}   \Psi\big[\exepi F^\ast \boxplus \exepi(I_C+T\circ G)^{\ast }\big] \\
 &=&      \Psi\Big[\exepi F^\ast \boxplus    \bigcup_{T\in\L_+(S,K)}  \exepi(I_C+T\circ G)^{\ast }\Big]  =  \A_2.
 \end{eqnarray*}
{(Note that the third equality comes from  \eqref{eqplusunion}).   }

{\bf (c)}  Firstly, {if}  $(C_1)$ and  $(C_2)$ hold,  we get from   {\bf (b)}
\begin{align}
\epi (F+I_A)^\ast&=\A_2= \bigcup_{T\in\L_+(S,K)}\Psi\Big(\exepi F^\ast \boxplus \exepi(I_C+T\circ G)^{\ast }\Big). \label{eq:33fbis}
\end{align}
{Now, for each $T\in\L_+(S,K)$, if $(C_3)$  holds then $(C_0)$ holds with $I_C$ and $T\circ G$ as well. Thus,  it follows from Lemma  \ref{thm:3.1aa}(ii) that}
 \begin{equation}
\label{eq:37f}
\Psi\left(\exepi F^\ast \boxplus \exepi(I_C+T\circ G)^{\ast }\right)=\Psi\Big(\exepi F^\ast \boxplus \exepi I_C^{\ast }\boxplus \exepi (T\circ G)^\ast\Big).
\end{equation}
{ Combining   \eqref{eq:33fbis} and \eqref{eq:37f}, one arrives at}
\begin{align*}
\epi (F+I_A)^\ast&= \bigcup_{T\in\L_+(S,K)}\Psi\Big(\exepi F^\ast \boxplus \exepi I_C^{\ast }\boxplus \exepi (T\circ G)^\ast\Big)\\
&=\Psi\left(\bigcup_{T\in\L_+(S,K)}[\exepi F^\ast \boxplus \exepi I_C^{\ast }\boxplus \exepi (T\circ G)^\ast]\right) \quad \textrm{(by \eqref{Psi-cup})}\\
&=\Psi\left(\exepi F^\ast \boxplus \exepi I_C^{\ast} \boxplus \bigcup_{T\in\L_+(S,K)}\exepi (T\circ G)^\ast\right) 
\quad \textrm{(by \eqref{eqplusunion})}\\
&=\Psi(\mathfrak{A}_3)=\A_3,
\end{align*}
{which completes the proof.}
\end{proof} 

\begin{remark} \label{rem73}
It is worth observing that when turning back to  the case {where} $Y=\mathbb{R}$ and $K=\mathbb{R}_+$,    { the cone  $\L_+(S,K)$  reduces to $S^+$,}  $\mathcal{P}_p(Y)={\mathbb{R}}$,  $K$-extended epigraph of  conjugate mappings  collapse to  usual  epigraphs of extended real-valued functions,  
 the $\boxplus$-sum now is the usual sum of two subsets in $X^*\times {\mathbb{R}}$ while the mapping $\Psi$   is nothing else but the identical mapping of $X^*\times {\mathbb{R}}$  (see  Remark \ref {rem_2dd}).  As a result,  (in this special case), {under the light of  Theorem   \ref{thm:4.1f}, } (and under some suitable regularity conditions), the sets {$\A_i$,}  $i = 1, 2, 3$,  go  back to the  known sets that represent $\epi (F+I_A)^* $  in the theory  of (scalar) convex optimization (see,    \cite{Bot2010,BGW09,DMVV-Siopt,DNV-08,DVN-08,Jeya 1,JSDL05}, 
     and references therein), and     as usual, in this case ($Y=\mathbb{R}$), we will use the lowercase letters for the extended real-valued functions (e.g., $f$,   $i_C, i_{-S}$ instead of $F$, $I_C$, $I_{-S}$):
  \begin{eqnarray*} 
 \A^\prime_1&:=&  \bigcup_{\lambda \in S^+} \epi (f+i_C+\lambda G)^\ast   = \epi (f+i_A)^*,\\
\A^\prime_2 &:=&  \epi f^\ast+\bigcup_{\lambda \in S^+} \epi (i_C+\lambda G)^\ast = \epi (f+i_A)^*, \ \ \textrm{and}  \\
\A^\prime_3  &:=& \epi f^\ast+\epi {i_C^\ast}+ \bigcup_{\lambda \in S^+} \epi (\lambda G)^\ast  = \epi (f+i_A)^*.
\end{eqnarray*} 

\end{remark} 
In the next two sections, we give some applications of the representations established in this section, firstly to establish characterizations of some equivalent forms of vector inequalities (also called vector Farkas-type results), and secondly to introduce variant forms of duality problems (called Lagrange and Fenchel-Lagrange dual problems) for vector optimization problems and  establish dual strong or stable dual strong duality results for these primal-dual pairs of problems. 

\section{{Characterizations of Vector Inequalities: Vector Farkas Lemmas }} \label{sec:5}

{Let $(F;G,C)$ be the triple that  defines the problem (VP) with its feasible set $A$ as in Section 4. For each $(L,y)\in \L(X,Y)\times Y$, we concern the vector inequality of the form:
\begin{equation*} 
F(x) - L(x)  \not<_K  - y, \ \ \   \forall x \in A, 
\end{equation*} 
which is equivalent to the inclusion: }

\begin{tabular}{c| c}
\begin{minipage}[t]{0.045\textwidth}
$( \alpha)$ 
\end{minipage}
&  
\begin{minipage}[t]{0.9\textwidth}
$x\in C, \; G(x)\in -S\ \  \Longrightarrow \ \ F(x)-L(x)+y\notin -\inte K$.
\end{minipage}
\end{tabular}

\noindent{We now  provide some necessary and sufficient conditions for $(\alpha)$ to hold. Each of such pair of equivalence  is often called a version of  the {\it vector Farkas lemma}. 
A version of vector Farkas lemma  is called {\it stable} if such an equivalence pair holds for  every $(L, y ) \in \L(X, Y)\times Y$. 
We start firstly with the general case. }

\subsection{General Vector  Inequalities}  

{Consider the 
 following  statements:  }

\medskip 
\noindent
\begin{tabular}{c |c}
\begin{minipage}[t]{0.045\textwidth}
$( \beta_1)$ 
\end{minipage}
 &  
\begin{minipage}[t]{0.9\textwidth}
$\exists T\in\L_+(S,K): (F+I_C+T\circ G)^\ast(L) \subset Y \textrm{ and }\\ y-(F+I_C+T\circ G)^\ast(L)\notin -\inte K$,
\end{minipage}
\end{tabular}

\medskip 

\noindent
\begin{tabular}{c |c}
\begin{minipage}[t]{0.045\textwidth}
$( \beta_2)$ 
\end{minipage}
 &  
\begin{minipage}[t]{0.9\textwidth}
$\exists L' \in \L(X,Y),\; \exists T\in \L_+(S,K):
F^\ast(L')\uplus (I_C+T\circ G)^\ast(L-L')\subset Y
\textrm{ and }\\
y-F^\ast(L')-(I_C+T\circ G)^\ast(L-L')\notin -\inte K$,
\end{minipage}
\end{tabular}

\medskip 
\noindent
\begin{tabular}{c| c}
\begin{minipage}[t]{0.045\textwidth}
$( \beta_3)$ 
\end{minipage}
 &  
\begin{minipage}[t]{0.9\textwidth}
$\exists L',L''\in \L(X,Y),\; \exists T\in \L_+(S,K): F^\ast(L')\uplus I_C^\ast(L'')\uplus (T\circ G)^\ast(L-L'-L'')\subset Y$ \textrm{ and} 
$y-F^\ast(L')-I_C^\ast(L'')-(T\circ G)^\ast(L-L'-L'')\notin -\inte K$,
\end{minipage}
\end{tabular}

\begin{theorem}[{Characterizations of stable  vector Farkas lemma}]
\label{thm:1}
For $i=1,2,3$, consider the  following statements:  

\noindent $({\bf a}_i)$\; $ \epi  (F+I_A)^\ast=\A_i,$

\noindent  $({\bf b}_i)$\;  $(\alpha)\Longleftrightarrow (\beta_i), \; \forall (L,y)\in \L(X,Y)\times Y$. 

\noindent Then, $[({\bf a}_i)   \Longleftrightarrow ({\bf b}_i)]$  for all $i=1,2,3$.
\end{theorem}

 \begin{proof}  We will give the proof for the most complicated  case, $i=3$, the conclusion for  other cases can be obtained by the same way.

$\bullet$ Take $(L,y)\in \L(X,Y)\times Y$.
According to \eqref{epiF-star}, 
\begin{equation}
\label{eq:41newe}
(\alpha)\; \Longleftrightarrow\;  (L, y)\in \epi (F+I_A)^\ast.
\end{equation}

 $\bullet$ {Now, we will show that 
\begin{align}
&(\beta_3)\; \Longleftrightarrow\; (L,y)\in \A_3. \label{eq:42newe}
\end{align}}

{\it Proof of  ``$\Rightarrow$'' in \eqref{eq:42newe}}.  Assume that $(\beta_3)$ holds.
Then, there {are $L',L''\in \L(X,Y)$ and $T\in \L_+(S,K)$} such that  $F^\ast(L')\uplus I_C^\ast(L'')\uplus (T\circ G)^\ast(L-L'-L'')\subset Y$ and 
$y\notin F^\ast(L')+ I_C^\ast(L'')+  (T\circ G)^\ast(L-L'-L'')-\inte K.$
On the other hand, by Proposition \ref{pro_decomp}(iii),
\begin{equation}
\label{eq:44_newe} 
\begin{aligned}
&F^\ast(L')+ I_C^\ast(L'')+  (T\circ G)^\ast(L-L'-L'')-\inte K \\
&= \wsup[ F^\ast(L')+ I_C^\ast(L'')+  (T\circ G)^\ast(L-L'-L'')]-\inte K\\
&=F^\ast(L')\uplus I_C^\ast(L'')\uplus  (T\circ G)^\ast(L-L'-L'')-\inte K.
\end{aligned}
\end{equation}
{Thus, $y\notin F^\ast(L')\uplus I_C^\ast(L'')\uplus  (T\circ G)^\ast(L-L'-L'')-\inte K$. It now follows from   Lemma \ref{pro_6hh}(i) and Proposition \ref{prop_1ab}(vi),  there exists
$V\in \P_p(Y)$ such that $ y\in F^\ast( L')\uplus I_C^\ast( L'') \uplus V$} and 
$ (T\circ G)^\ast( L-L'-L'')\preccurlyeq_K V$. By letting  $ U=F^\ast( L')\uplus I_C^\ast( L'') \uplus V \in \P_p(Y)$, one gets $ y\in  U$ and 
\begin{equation}\label{eq:lxab}
\begin{aligned}
( L, U)&=(L' + L''+ L-L'-L'',F^\ast( L')\uplus I_C^\ast( L'') \uplus V)\\
&=(L',F^\ast(L'))\boxplus (L'',I_C^\ast(L''))\boxplus (L-L'-L'', V)\\
&\in \exepi F^\ast \boxplus \exepi{I^\ast_C} \boxplus \exepi ( T\circ G)^\ast\subset \mathfrak{A}_3.
\end{aligned}
\end{equation}
Hence, $(L,y)\in\Psi(\mathfrak{A}_3)=\A_3$.

{\it Proof of``$\Leftarrow$'' in \eqref{eq:42newe}.} Assume that $(L,y)\in \A_3$. Then   there exists $U\in \mathcal{P}_p(Y)$ such that $(L,U)\in \mathfrak{A}_3$ and $y\in U$. As $(L,U)\in \mathfrak{A}_3$, there are $(L',U')\in \exepi F^\ast$, $(L'',U'')\in \exepi {I^\ast_C} $, $(L-L'-L'',W)\in\exepi (T\circ G)^\ast$ such that $U=U'\uplus U''\uplus W$. Then,
$F^\ast(L')\preccurlyeq_K U'$, ${I^\ast_C} (L'')\preccurlyeq_K U''$, $(T\circ G)^\ast(L-L'-L'')\preccurlyeq_K W$, and consequently, 
\begin{equation}\label{eq32--bis}
M:=F^\ast(L')\uplus {I^\ast_C} (L'')\uplus (T\circ G)^\ast(L-L'-L'') \preccurlyeq_K U
\end{equation}
(by Proposition \ref{prop_1ab} (v)). This yields $M\ne\{+\infty_Y\}$, which, together with the fact that $M$ belong to $\P_p(Y)^\infty$, show that $M\subset Y$. 
So, according to Proposition \ref{prop_4gg} (i), \eqref{eq32--bis} entails
$ U\cap (M-\inte K)=\emptyset$
showing that $y\notin M-\inte K$ (as $y\in U$) which also means that $(\beta_3)$ holds.

\medskip

 {Finally, taking  \eqref{eq:41newe} and \eqref{eq:42newe} into account, one gets}
\begin{align*}
({\bf a}_3) \;  &\Longleftrightarrow \;  [ (\alpha) \Leftrightarrow (\beta_3),\;\forall (L,y)\in  \L (X, Y)   \times Y]\\
&\Longleftrightarrow \; ({\bf b}_3), 
\end{align*}
which is desired.
   \end{proof}

\begin{remark}
{According to Lemma \ref{rem:3b}, $({\bf a}_1) \Leftrightarrow ({\bf b}_1)$ is  \cite[Proposition 5.1]{DGLL17} ((a)$\Leftrightarrow$(c)), when $\V=\L(X,Y)$ and $\mathcal{W}=Y$. To the best knowledge of the authors, the characterizations of vector Farkas lemmas given in  Theorem \ref{thm:1} with  $i=2,3$ are all new. Moreover, taking Remark \ref{rem73} into account, we will see that Theorem  \ref{thm:1}  when specifying to the scalar case ($Y = \mathbb{R}$ and $K= \mathbb{R}_+$),  extends some known  results in the literature, such as the ones in  \cite{DNV-08,DVN-08},  
as shown in the next corollary (see also Remark \ref{rem:5.2d}). }
   \end{remark}

Let us recall that $ \A'_i$, $i=1, 2, 3$,  are the sets defined in  Remark \ref{rem73}. 

\begin{corollary}  \label{cor:5.1}
Let $f\colon X\rightarrow \overline{\mathbb{R}}$ be a proper function, $
G\colon X\rightarrow Z\cup\{+\infty_{Z}\}$ be a proper mapping, and  $C$ be a
nonempty convex subset of $X$ such that  $A\cap \dom F\neq\emptyset$ (where $A:=C\cap G^{-1}(-S)$).  Consider the following statements: 

\begin{tabular}{c c}
\begin{minipage}[t]{0.045\textwidth}
\indent $( \alpha')$ 
\end{minipage}
&  
\begin{minipage}[t]{0.9\textwidth}
$x\in C, \; G(x)\in -S \Longrightarrow f(x)-x^\ast(x)\ge r$,
\end{minipage}
\end{tabular}

\begin{tabular}{c c}
\begin{minipage}[t]{0.045\textwidth}
$( \beta'_1)$ 
\end{minipage}
 &  
\begin{minipage}[t]{0.8\textwidth}
$\exists \lambda\in S^+ : (f+i_C+\lambda G)^\ast(x^\ast) \le -r$,
\end{minipage}
\end{tabular}

\begin{tabular}{c c}
\begin{minipage}[t]{0.045\textwidth}
$( \beta'_2)$ 
\end{minipage}
 &  
\begin{minipage}[t]{0.8\textwidth}
$\exists u^\ast \in X^\ast,\; \exists \lambda\in S^+: f^\ast(u^\ast) + (i_C+\lambda G)^\ast( x^\ast-u^\ast) \le -r$,
\end{minipage}
\end{tabular}

\begin{tabular}{cc}
\begin{minipage}[t]{0.045\textwidth}
$( \beta'_3)$ 
\end{minipage}
 &  
\begin{minipage}[t]{0.8\textwidth}
$\exists u^\ast,v^\ast\in X^\ast,\; \exists \lambda \in S^+: f^\ast(u^\ast) + i^\ast_C(v^\ast) + (\lambda G)^\ast (x^\ast-u^\ast-v^\ast)\le -r$. 
\end{minipage}
\end{tabular}
 
\noindent {Moreover,   let }

 $({\bf a}'_i)$\; {$ \epi  (f+i_A)^\ast= \A'_i$}, 

 {$({\bf b}'_i)$\;} $(\alpha')\Longleftrightarrow (\beta'_i), \; \forall (x^\ast,y)\in {X^\ast}\times \mathbb{R}$. 

\medskip 
\noindent {Then, $[({\bf a}'_i)   \Leftrightarrow {({\bf b}'_i)}]$ for all $i=1,2,3$.}

\end{corollary} 

\begin{proof}
{When $Y=\mathbb{R}$ and $K=\mathbb{R}_+$, one has, for each $i=1,2,3$,  $\A_i$ collapses  to $\A'_i$  (see Remark  \ref{rem73}), $(\alpha)$ becomes  $(\alpha')$, while  $(\beta_i)$ becomes  $(\beta'_i)$.}     The equivalences $[({\bf a}'_i)   \Leftrightarrow ({\bf b}'_i)]$, $i=1,2,3$ now follow immediately from Theorem \ref{thm:1}.
\end{proof}

\begin{remark}\label{rem:5.2d}  {In  the case where $Y = \mathbb{R}$ and $K = \mathbb{R}_+$, if we   assume further  that  $C$ is a closed convex subset of $X$, that $f$ is a proper convex and lower semicontinuous function,   and that $G$ is s proper $S$-convex and positively $S$-lsc mapping. Then, according to  \cite[Section 8]{Bot2010}, one has
$$\epi (f+i_A)^\ast=\cl[ \epi f^*+\epi i^\ast_C+\bigcup_{z^*\in S^+}\epi(z^*\circ G)^*].$$
So, in  this setting, observe firstly that Corollary \ref{cor:5.1} with  $i=1$ is \cite[Corollary 6.14]{DM-12}. Moreover, 
note also that $({\bf a}'_3)$ is equivalent to the condition stating  that  the set $\epi f^*+\epi i^\ast_C+\bigcup_{z^*\in S^+}\epi(z^*\circ G)^*$ is weak*-closed,  which is exactly the condition $({\bf CC})$ introduced in \cite{DNV-08,DVN-08}. Moreover, by Proposition \ref{pro:9}, the fulfillment of $({\bf a}'_3)$  ensures that both $({\bf a}'_1)$ and $({\bf a}'_2)$  hold. So, Corollary \ref{cor:5.1} can be considered  as  {extended} versions   of  \cite[Corollary 5.2]{DNV-08} and \cite[Corollary 6.2]{DVN-08},   and as a result, Theorem \ref{thm:1} is vector extension version of the Farkas lemmas just mentioned in \cite{DNV-08} and \cite{DVN-08} in twofolds: Firstly, the Farkas versions in Theorem \ref{thm:1} extend the Farkas-type result in the two mentioned paper from scalar systems to systems involving vector-valued functions; secondly, they extend the Farkas-type result to stable Farkas-type results.  } 
\end{remark}

\subsection{ Convex Vector inequalities}

\begin{corollary}[Convex vector Farkas lemma]  
\label{cor:5.3}
Let  $F$ be a proper and $K$-convex mapping,   $G$ be a proper and $S$-convex mapping,  and  $C$ be   a  nonempty and convex subset of $X$.  The following statements  hold:
\begin{itemize}
\item[{\rm (i)}] If $(C_1)$ holds then ${\bf (b_1)}$ holds, 
\item[{\rm (ii)}] If $(C_1)$ and $(C_2)$ hold then   ${\bf (b_2)}$ holds, 
\item[{\rm (iii)}] If $(C_1)$, $(C_2)$ and $(C_3)$ hold then   ${\bf (b_3)}$ holds,
\end{itemize}
{where  ${\bf (b_i)}$, $i = 1, 2, 3$ are statements in Theorem \ref{thm:1}.} 
\end{corollary}

\begin{proof}
{These are direct consequences of  Theorems \ref{thm:4.1f} and  \ref{thm:1}.} 
\end{proof}

{Turning  to the scalar case,  i.e.,  $Y=\mathbb{R}$ and $K=\mathbb{R}_+$,  the conditions $(C_1)$ and $(C_2)$ collapse, respectively, to the ones below: }

\noindent
\begin{tabular}{c  c}
\hskip1cm $(C_1^\prime)$ &  
\begin{minipage}{0.9\textwidth}
$\exists x_1 \in  C\cap \dom f: G(x_1)\in -\inte S$
\end{minipage}
\end{tabular}

\noindent
\begin{tabular}{c  c}
\hskip1cm $(C^\prime_2)$ &  
\begin{minipage}{0.9\textwidth}
$\exists  x_2\in  C\cap\dom G:  f\textrm{ is continuous  at } x_2$.
\end{minipage}
\end{tabular}
The next corollary is a direct sequence of Corollary \ref{cor:5.3}.

\begin{corollary}[Convex Farkas-type results]  
Let $f\colon X\rightarrow \overline{\mathbb{R}}$ be a  proper convex function, $
G\colon X\rightarrow Z\cup\{+\infty_{Z}\}$ be a proper $S$-convex mapping, and $C$ be a nonempty convex subset of $X$ such that  $A\cap \dom F\neq\emptyset$ (where $A:=C\cap G^{-1}(-S)$).  
{The following statements hold:}
\begin{itemize}
\item[{\rm (i)}] If $(C^\prime_1)$ holds then  ${\bf (b^\prime_1)} $ holds,  
\item[{\rm (ii)}] If $(C^\prime _1)$ and $(C^\prime_2)$ hold then  ${\bf (b^\prime_2)} $ holds,  
 \item[{\rm (iii)}] If $(C^\prime_1)$, $(C^\prime_2)$ and $(C_3)$ hold then  ${\bf (b^\prime_3)} $ holds, 
 \end{itemize}
 {where  ${\bf (b^\prime_i)} $, $i = 1, 2, 3$ are statements in   Corollary \ref{cor:5.1}. }  \end{corollary}

\section{Lagrange and Fenchel-Lagrange  Duality for  Vector Optimization Problems}
We retain the notation   in Section 4 and 
consider a  vector optimization problem 
\begin{align*}
({\rm VP})\quad &\winf \{F(x): x\in C,\; G(x)\in -S\}  
\end{align*}
{with the feasible  set $A := C \cap G^{-1}(-G)$ being non-empty. }

The \emph{Lagrange dual problem} 
 of $({\rm CVP})$ is defined  in \cite{DL2017} as: 
\begin{align*}
({\rm VD}_{1}) \quad&  \mathop{\wsup}_{T\in \L_+(S,K)} \hskip1em\mathop{\winf}_{x\in C}\hskip0.7em [F(x)+(T\circ G)(x)],   
\end{align*}
{or, equivalently, 
\begin{align*}
({\rm VD}_{1}) \quad&  \mathop{\wsup}_{T\in \L_+(S,K)} -(F+I_C+T\circ G)^\ast(0_\L).
\end{align*} }
We now propose some two  {new} types of \emph{``Fenchel-Lagrange" dual problems} of $({\rm VP})$: 
\begin{align*}
&({\rm VD}_{2})&& \mathop{\wsup}_{\substack{L'\in\L(X,Y)\\T\in \L_+(S,K)}} \hskip0.6em  -[F^*(L')\uplus (I_C+T\circ G)^*(-L')], \\
&({\rm VD}_{3})&& \mathop{\wsup}_{\substack{L',L'' \in\L(X,Y)\\ T\in\L_+(S,K)}} \!\! -[F^*(L')\uplus I^\ast_C(L'')\uplus  (T\circ G)^*(-L'-L'')]. 
\end{align*}
It is worth observing that when going back to the scalar problem, i.e., when $Y = \mathbb{R}$ and (VP) is a scalar problem (P),  the last two dual problems turn back to the Fenchel-Lagrange dual problems  $({\rm D}_{2})$   and  $({\rm D}_{3})$ of (P) mentioned in the Introduction (Section 1), and      this justifies the names of these dual problems.



\begin{definition}
We say that ``strong duality holds for the pair $({\rm VP})-({\rm VD}_1)$'' if 
$\winf ({\rm VP})=\wmax ({\rm VD}_1).$\footnote{Observe  that when $\winf ({\rm VP})=\wmax ({\rm VD}_1)$, one has  $\wsup ({\rm VD}_1)=\wmax ({\rm VD}_1)$ (see  {Proposition \ref{pro_5hh}  
(i)}), and hence, $({\rm VD}_1)$ attains  at any value from $\wsup ({\rm VD}_1)$.}

We denote by $({\rm VP}^L)$  the problem $({\rm VP})$ perturbed by a linear operator $L\in \L(X,Y)$, 
\begin{align*}
({\rm VP}^L)\quad &\winf \{F(x)-L(x): x\in C,\; G(x)\in -S\}.
\end{align*}
{Then the  Lagrange dual problem}  of  $({\rm VP}^L)$ will be denote by  $({\rm VD}_1^L)$.
We say  that  the {\it stable strong duality holds for the pair  $({\rm VP})-({\rm VD}_1)$} if the strong duality holds for the  pair $({\rm VP}^L)-({\rm VD}_1^L)$ for any  $L\in \L(X,Y)$. 
The notions stable strong duality corresponding to the other pairs of primal-dual problems $({\rm VP})-({\rm VD}_2)$ and $({\rm VP})-({\rm VD}_3)$ will be understood in the same way.
\end{definition}

\begin{lemma}\label{lem:6.2c}
$\wsup ({\rm VD}_{i}^L)=\wsup \{-y: (L,y)\in \A_i\}$ for all $L\in \L(X,Y)$  and  $i\in\{1,2,3\}$, where $\A_i $ are the sets defined in \eqref{eq21}. 

\end{lemma}

\begin{proof}
We prove the conclusion for $i = 3$, i.e.,   that {$\wsup ({\rm VD}_{3}^L)=\wsup\{-y: (L,y)\in \A_3\}$}. 
The proofs of other cases are  similar. 
On the one hand, one has
\begin{align*}
\wsup({\rm VD}_{3}^L)&=\mathop{\wsup}_{\substack{L',L''\in\L(X,Y)\\T\in \L_+(S,K)}}  -[(F-L)^*(L') \uplus I_C^\ast(L'')\uplus(T\circ G)^*(-L'-L'')]\\
&=\mathop{\wsup}_{\substack{L',L''\in\L(X,Y)\\T\in \L_+(S,K)}}  -[F^*(L+L')  \uplus I_C^\ast(L'')\uplus(T\circ G)^*(-L'-L'')]\\
&=\mathop{\wsup}_{\substack{L',L''\in\L(X,Y)\\T\in \L_+(S,K)}}  -[F^*(L+L')  \uplus I_C^\ast(L'')\uplus(T\circ G)^*(-L'-L'')-K]\\
&=\mathop{\wsup}  \left\{-y: 
\begin{array}{r}y\in F^*(L+L')  \uplus I_C^\ast(L'')\uplus(T\circ G)^*(-L'-L'')+K\\ {\textrm{ for some }} L',L''\in\L(X,Y)\textrm{ and }T\in \L_+(S,K)
\end{array}
\right \}
\end{align*}
\noindent (where the third equality follows from Proposition \ref{pro_decomp}(vii)).
{On the other hand, as  $F^*(L+L') \uplus I_C^\ast(L'')\uplus(T\circ G)^*(-L'-L'')\in \P_p(Y)$, one has }
\begin{eqnarray}\label{eq:azaza}
\begin{aligned}
& \begin{array}{l} y\in F^*(L+L') \uplus I_C^\ast(L'')\uplus(T\circ G)^*(-L'-L'')+K\\
\hskip1cm\textrm{for some } L',L''\in\L(X,Y)\textrm{ and }T\in \L_+(S,K)\end{array}\\
 \Longleftrightarrow\; & \begin{array}[t]{l}  y\notin F^*(L+L') \uplus I_C^\ast(L'')\uplus(T\circ G)^*(-L'-L'')-\inte K\\
\hskip1cm\textrm{for some } L',L''\in\L(X,Y)\textrm{ and }T\in \L_+(S,K)\quad \textrm{(by Lemma \ref{pro_6hh}(i))}\end{array}     \\
 \Longleftrightarrow\; & \begin{array}[t]{l}  y\notin F^*(L+L') + I_C^\ast(L'')+(T\circ G)^*(-L'-L'')-\inte K\\
\hskip1cm\textrm{for some } L',L''\in\L(X,Y)\textrm{ and }T\in \L_+(S,K)\quad\textrm{(by Prop. \ref{pro_decomp}(iii))}\end{array}\\ 
 \Longleftrightarrow\; &(\beta_3) \textrm{ holds }\\ 
  \Longleftrightarrow\;  & (L,y)\in \A_3 \quad\textrm{ (see \eqref{eq:42newe})}.   
\end{aligned}
\end{eqnarray}
Consequently, {$\wsup({\rm VD}_{3}^L)=\wsup \{-y: (L,y)\in \A_3\}$} and we are done.
\end{proof}

\begin{proposition}[Weak duality]
 \label{thm_WD}
For any $L\in \L(X,Y)$, it holds 
$$\wsup ({\rm VD}_3^L)\preccurlyeq_K\wsup ({\rm VD}_2^L)\preccurlyeq_K
\wsup ({\rm VD}_1^L)\preccurlyeq_K\winf ({\rm VP}^L).$$
\end{proposition}

\begin{proof} 
Recall that  $({\rm VD}_{1})$ is nothing else but  the problems $({\rm DVP})$ in \cite{DL2017}, hence,
{the inequality $ \wsup ({\rm VD}_1^L)  \preccurlyeq_K  \winf ({\rm VP}^L)$ } was established in  \cite[Theorem 5]{DL2017} (specialize for the case when the uncertain set is a singleton). 

{Now, by Lemma  \ref{lem:6.2c}  one has $\wsup({\rm VD}_{i}^L)=\wsup \{-y: (L,y)\in \A_i\}$ for each $i = 1, 2, 3$.  
Moreover, by Proposition \ref{pro:9}, 
$
\A_3 \subset  \A_2  \subset  \A_1$, and the conclusion follows as   $\wsup M \preccurlyeq_K\wsup N$ if   $M,N\subset Y^\bullet$ and $M\subset N$  (see Proposition \ref{prop_4gg}(iii)). }
 \end{proof}

{The main result of this section  on the stable strong duality for (VP) is given in the next theorem.  }

\begin{theorem}[Principle for stable strong duality  of  (VP)]
\label{thm:6.2}
For $i=1,2,3$, consider the  following statements:  

\noindent $({\bf a}_i)$\; {$\epi  (F+I_A)^\ast=\A_i,$}


\noindent  $({\bf c}_i)$\; {The  stable strong duality holds for the pair  $({\rm VP})-({\rm VD}_{i})$.}


\noindent Then, {$[({\bf a}_i)   \Leftrightarrow ({\bf c}_i)]$  for all $i = 1, 2, 3$.}
\end{theorem}

\begin{proof}  We prove   $[({\bf a}_3)\!   \Leftrightarrow \!({\bf c}_3)] $ only. 
The proofs for the case $i = 1, 2$  are similar.
Take  $L\in \L(X,Y)$. For each $L',L''\in\L(X,Y)$ and $T\in \L_+(S,K)$,  let 
\begin{eqnarray*}
\mathcal{K}_L(L',L'',T)&:=& -F^*(L+L')\uplus I_C^\ast(L'')\uplus (T\circ G)^*(-L'-L'')], \\
N_L&:=& \bigcup_{\substack{L',L''\in \L(X,Y)\\T\in \L_+(S,K) }} \mathcal{K}_L(L',L'',T). 
\end{eqnarray*}
One then  has 
\begin{equation} \label{eq:abxx}     
\begin{array}{lll}
{\wsup ({\rm VD}_{3}^L)} &=& \mathop{\wsup}\limits_{\substack{L',L''\in\L(X,Y)\\T\in \L_+(S,K)}}  \mathcal{K}_L(L',L'',T)=\wsup N_L\\
{\wmax ({\rm VD}_{3}^L)}&=& \mathop{\wmax}\limits_{\substack{L',L''\in\L(X,Y)\\T\in \L_+(S,K)}}  \mathcal{K}_L(L',L'',T)=\wmax N_L. 
\end{array}
\end{equation}

 $[({\bf a}_3)   \Rightarrow{({\bf c}_3)}]$  {  Assume that $({\bf a}_3)$ holds. Take  $L\in \L(X,Y)$, and we will show that }
\begin{equation}\label{eq_26f}
{\winf ({\rm VP}^L)=\wmax ({\rm VD}_{3}^L).}
\end{equation}

 $\bullet$ {\it Step 1.} As $({\rm VP})$ is feasible,   ${\winf ({\rm VP}^L)}\ne\{+\infty_Y\}$. 
If  ${\winf ({\rm VP}^L)}=\{-\infty_Y\}$  then, by Proposition \ref{thm_WD},   ${\wsup ({\rm VD}_{3}^L)}=\{-\infty_Y\}$, and so,  
$\mathcal{K}_L(L',L'',T)=\{-\infty_Y\}$ for all $L',L''\in\L(X,Y)$ and $T\in \L_+(S,K)$. Consequently,  
${\wmax ({\rm VD}_{3}^L)=\{-\infty_Y\}=\winf ({\rm VP}^L)}, $
and \eqref{eq_26f} holds.

 $\bullet$ {\it Step 2.}
Assume from now that $\winf ({\rm VP}^L) \subset Y$.
As  $\winf ({\rm VP}^L)\in \P_p(Y)$, one has 
\begin{align}\label{eqthm61a}
\winf ({\rm VP}^L)&=\wsup[\winf ({\rm VP}^L)] \qquad\qquad\; \textrm{(see {Lemma} \ref{pro_6hh}(iii))}\nonumber   \\
&=\wsup[\winf \{F(x)-L(x): x\in A\}]\quad   \nonumber   \\
&=\wsup[-(F+I_A)^\ast(L)]\\
&=\wsup[-(F+I_A)^\ast(L)-K] \quad \textrm{(see Proposition \ref{pro_decomp}(vii))}   \nonumber \\
&=\wsup \{-y: (L,y)\in \epi (F+I_A)^\ast\}.  \nonumber 
\end{align}
On the other hand, as  $({\bf a}_3)$ holds,  $\epi  (F+I_A)^\ast=\A_3$, and by Lemma   \ref{lem:6.2c},  $\wsup({\rm VD}_{3}^L)=\wsup \{-y: (L,y)\in \A_3\}$, which,   together with  \eqref{eqthm61a}  and \eqref{eq:abxx},  yields   
\begin{equation}\label{eq:34new}
\winf ({\rm VP}^L)=\wsup({\rm VD}_{3}^L)=\wsup N_L.
\end{equation}
\indent$\bullet$ {\it Step 3.} We now prove  that $\winf  ({\rm VP}^L)\subset N_L$.
Take  $y\in \winf  ({\rm VP}^L)=\winf [(F-L)(X)]$.
It then follows from  Proposition \ref{pro_decomp}(iv)     (see also,   Remark \ref{rem_1hh})   
that $y\notin (F-L)(X)+\inte K$, which is 
 equivalent to  
$$ F(x)-L(x)-y\notin -\inte K, \ \ \forall   x\in A.$$ 
This  is nothing else but  $(\alpha)$ in Section \ref{sec:5} with $y$ being replaced by $-y$. As $({\bf a}_3)$ holds, it follows from   the  stable  vector Farkas lemma (Theorem \ref{thm:1} with $i = 3$) that there are 
$ \bar L',\bar L'' \in \L(X,Y)$ and  $\bar T\in \L_+(S,K)$ such that  
\begin{equation}\label{eq:35new}
F^\ast(L+L')\uplus I_C^\ast(L'')\uplus (T\circ G)^\ast(-L'-L'')\subset Y
\end{equation}
 and $-y\notin F^\ast(L+\bar L')+I_C^\ast(\bar L'')+ (\bar T\circ G)^\ast(-\bar L'-\bar L'')  -\inte K$,  which means that 
\begin{equation}\label{eq:36new}
-y\notin -\mathcal{K}_L(\bar L',\bar L'',\bar T)  -\inte K. 
\end{equation}
On the other hand, as   $ y\in \winf  ({\rm VP}^L) = \wsup N_L$ (see \eqref{eq:34new}), one has  $ y \not<_K y^\prime $ for all $y^\prime \in N_L$, yielding  
\begin{equation}\label{eq:37new}
 y\notin \mathcal{K}_L(\bar L',\bar L'',\bar T)-\inte K.
\end{equation}
Now, as $\mathcal{K}_L(\bar L',\bar L'',\bar T)\in \P_p(Y)$ (see \eqref{eq:35new}), the 
three sets $\mathcal{K}_L(\bar L',\bar L'',\bar T)-\inte K$, $\mathcal{K}_L(\bar L',\bar L'',\bar T)$, and  $\mathcal{K}_L(\bar L',\bar L'',\bar T)+\inte K$ constitute a decomposition of $Y$, which, together with \eqref{eq:36new} and \eqref
{eq:37new}, yields 
 $y\in \mathcal{K}_L(\bar L',\bar L'',\bar T)$.  This shows that $y\in N_L$, and hence, $\winf M_L\subset N_L$.

$\bullet$ {\it Step 4.} It now follows from Step 3 and Step 4 that  (see also \eqref{eq:abxx})
$$\winf ({\rm VP}^L)=N_L\cap\wsup N_L=\wmax N_L=\wmax ({\rm VD}_{3}^L), $$
which is \eqref{eq_26f}.

\medskip

 $[{({\bf c}_3)}   \Rightarrow ({\bf a}_3)]$ Assume that ${({\bf c}_3)}$ holds, i.e., \eqref{eq_26f} holds for all $L\in {\L(X,Y)}$. {Taking   Proposition \ref{pro:9} into account, to prove $({\bf a}_3)$,  it suffices  to show that
\begin{equation}\label{eq_35iii}
\epi (F+I_A)^*\subset \A_3.
\end{equation}
Take $(L,y)\in \epi (F+I_A)^*$. Then
 \begin{equation*}
y\in   \wsup [(L-F)(A)] + K = - \winf (F-L)(A)+K=-{\winf ({\rm VP}^L)}+K.
\end{equation*}}
On the other hand, as $({\bf c}_3)$ holds,   one also has 
$$\winf ({\rm VP}^L)=\wmax ({\rm VD}^L_{3})=\wmax N_L\subset N_L. $$
So,   $y\in -N_L+K$ and thus,   
there are $\bar L',\bar L''\in \L(X,Y)$ and $\bar T\in \L_+(S,K)$ such that 
\begin{align*}
&y\in  -\mathcal{K}_L(\bar L',\bar L'',\bar T) + K\\
\Longleftrightarrow\; &y\in F^*(L+\bar L')\uplus I_C^\ast(\bar L'')\uplus (\bar T\circ G)^*(-\bar L'-\bar L'')] +K\\
\Longleftrightarrow\;& (L,y)\in \A_3\quad \textrm{(see \eqref{eq:azaza})}.
\end{align*}
Consequently,  $(L,y)\in \epi (F+I_A)^*$ then  
$(L, y) \in \A_3  $  and  \eqref{eq_35iii} follows. The proof is complete.
\end{proof}

\begin{corollary}[Stable strong duality I]\label{cor:6.6}
Assume that  $F$ is $K$-convex,  that $G$ is $S$-convex,  and that $C$ is convex. Then, the following statements are holds true:

\begin{itemize}
\item[{\rm (i)}] {If $(C_1)$ holds then the  stable strong duality holds for the pair  $({\rm VP})-({\rm VD}_{1})$.}

\item[{\rm (ii)}] {If $(C_1)$ and $(C_2)$ hold then the  stable strong duality holds for pairs  $({\rm VP})-({\rm VD}_{1})$ and $({\rm VP})-({\rm VD}_{2})$.}

\item[{\rm (iii)}] {If $(C_1)$, $(C_2)$ and $(C_3)$ hold then the stable strong duality holds for three pairs  $({\rm VP})-({\rm VD}_{i})$, $i=1,2,3$.}
\end{itemize}

\end{corollary}

\begin{proof}
It follows from Theorems \ref{thm:6.2} and \ref{thm:4.1f}.
\end{proof}

 \begin{remark}  
Theorem \ref{thm:6.2} and Corollary \ref{cor:6.6} for the case $i=1$ are stable versions of \cite[Corollaries 1 and 3]{DL2017} {(with the uncertainty set being a singleton). The results for other cases (i.e., $i = 2, 3$),  up to the best knowledge of the authors, are  new.}  Consequently,    Theorem \ref{thm:6.2} and Corollary \ref{cor:6.6} with $i = 2, 3$,    are probably the first version of    Fenchel-Lagrange duality results  for vector problems which  extend
 the same kinds of duality (see     \cite{Bot2010} and \cite{DNV-08}, and references therein)    for scalar  to     vector optimization problems (see  Corollaries 6.2 and 6.3 below). 
\end{remark}

 It is worth observing  that when   $Y=\mathbb{R}$ and $K=\mathbb{R}_+$ the problem (VP) becomes (P), 
while the \emph{dual problems} {$({\rm VD}_{i})$},   $i=1,2,3$,  collapse to the  dual problem ${({\rm D}_1)} $, ${({\rm D}_2)} $, and ${({\rm D}_3)} $
in Section 1 (Introduction). 
{Note that ${({\rm D}_1)}$ and  ${({\rm D}_2)}$ are $(D^{C_L})$ and $(D^{C_{FL}})$, respectively, in \cite{Bot2010} while ${({\rm D}_1)}$ and  ${({\rm D}_3)}$ are  the problems $({\rm LDQ})$ and $({\rm FLDQ})$ proposed in  \cite{DNV-08}  for the case when $Y = \mathbb{R}$ and $K = \mathbb{R}_+$. } {The next two corollaries are stable versions (extensions) of the corresponding results in   \cite{Bot2010}  and in \cite{DNV-08}.  
 They are direct consequences of Theorem \ref{thm:6.2} and Corollary \ref{cor:6.6}, respectively.  }

\begin{corollary}\label{cor:6.3d}   Let $f : X \rightarrow \mathbb{R}  \cup \{ + \infty\}$, $A^\prime _i$ be the sets defined in Remark  \ref{rem73}, and $i_C$ be the usual indicator function  of the set $C$. 
For each   $i=1,2,3$,   consider the following statements:

\noindent $({\bf a}'_i)$\; {$ \epi  (f+i_A)^\ast= \A'_i$},

\noindent ${({\bf c}'_1)}$\; $\inf(f-x^\ast)(A)=\max_{\lambda\in S^+} \inf (f-x^\ast+\lambda G)(C),$ {$\forall x^\ast\in X^\ast$},  

\noindent ${({\bf c}'_2)}$\; $\inf(f-x^\ast)(A)=\max_{\substack{u^\ast \in X^\ast\\\lambda\in S^+}} [-f^*(u^\ast)-(i_C+\lambda G)^*(x^\ast-u^\ast)],$  {$\forall x^\ast\in X^\ast$},  

\noindent ${({\bf c}'_3)}$\; $\inf(f-x^\ast)(A)=\max_{\substack{u^\ast\!\!,v^\ast\in X^\ast\\\lambda\in S^+}}  \! [-f^*(u^\ast)-{i^\ast_C}(v^\ast)- (\lambda G)^*(x^\ast-u^\ast-v^\ast)],$ {$\forall x^\ast\in X^\ast$}. 
\noindent Then, $[({\bf a}'_i)   \Leftrightarrow {({\bf c}'_i)}]$  {for each    $i = 1, 2, 3$.}
\end{corollary}

\begin{corollary}\label{cor:6.4c}
Assume that  {$f$ is convex},  that $G$ is $S$-convex,  and that $C$ is convex.
{The following statements  hold:}

\begin{itemize}
\item[{\rm (i)}] If $(C_1')$ holds then ${({\bf c}'_1)}$   holds,
\item[{\rm (ii)}] If $(C_1')$ and $(C_2')$ hold then ${({\bf c}'_1)}$ and  ${({\bf c}'_2)}$  hold,

\item[{\rm (iii)}] If $(C_1')$, $(C_2')$ and $(C_3 )$ hold then  ${({\bf c}'_1)}$, ${({\bf c}'_2)}$, and ${({\bf c}'_3)}$ hold,
\end{itemize}
{where ${({\bf c}'_1)}$, ${({\bf c}'_2)}$, and ${({\bf c}'_3)}$ are statements   in Corollary \ref{cor:6.3d},  and $(C_1')$,  $(C_2')$ are regularity conditions introduced  in Subsection 5.2.} \end{corollary}

\begin{remark}
\label{rem:6.1e}    
 (a)  The condition $(C'_1)$ and $(C'_2)$ are nothing else but the condition $(RC_1^L)$ and $(RC_1^{FL})$ in \cite{Bot2010}, and hence, Corollary \ref{cor:6.4c} (i)-(ii) extends and can be considered as ``stable strong duality version" of  the strong duality     given in \cite[Theorems 3.4 and 3.6]{Bot2010}.

(b)  
Recall that when $C$ is a closed convex subset of $X$, $f$ is a proper convex and {lsc  function,  } and $G$ is a proper $S$-convex and positively $S$-lsc mapping, $({\bf a}'_3)$ is equivalent to  $({\bf CC})$ in \cite{DNV-08,DVN-08}, and that $({\bf a}'_1)$ and $({\bf a}'_2)$ hold whenever $({\bf a}'_1)$ holds (see Remark \ref{rem:5.2d}).
 So, Corollary \ref{cor:6.3d} covers  \cite[Corollaries 4.5, 4.6, 4.7]{DNV-08} and \cite[Corollaries 6.4, 6.5]{DVN-08} (non-stable version) and \cite[Theorem 6.2, 6.3]{DVN-08} (stable version).

{(c) 
Under the convex and closedness assumptions  as in (b),} $({\bf a}'_1)$ is equivalent to the fact that  $\bigcup_{z^*\in S^+}\epi(f+i_C+z^*\circ G)^*$ is weak*-closed, and hence, Corollary \ref{cor:5.1}, together with Corollary \ref{cor:6.3d}, for the case $i=1$ returns to \cite[Corollary 5]{DVV-14}.

(d) Lastly, {it is worth emphasizing that (to the best knowledge of the authors)  no  results in duality for vector problems existed in the literature could cover the ones   in } $({\bf c}'_2)$, $({\bf c}'_3)$ (when turning back to the case  $Y = \mathbb{R}$).      
\end{remark}
\appendix


\section{  {Proof} of the Basic Lemma 1 (Lemma \ref{exepi})} 


Let us set
\begin{align}\label{eq:19a}
	\Delta&:=\bigcup_{\substack{x\in \dom F_1\\x'\in \dom F_2}}
		[ M+ \bar L(x)-F_1(x)-F_2(x')-K] \times \{x-x'\}.
\end{align}

$\bullet$ {\it Step 1.}  {\it  We firstly prove  that  there is $\widehat y\in Y$ such that $ ( \widehat y-\inte K)\!\times\! \{0_X\} \subset \inte \Delta$.} 
Pick $\bar k\in \inte K$, it is obvious that   $F_2(\widehat x) +\bar k - K\in \mathcal{N}_Y(F_2(\widehat x))$. So, by $(C_0)$, there is  $U\in \mathcal{N}_X(0_X)$ such that $F_2(\widehat x+ U) \subset F_2(\widehat x)+\bar k- K$ which leads to 
\begin{equation*}
F_2(\widehat x)+\bar k\in F_2(\widehat x+ u)+K, \quad\forall u\in U.
\end{equation*}
 It results that if we take $\widehat y:= m+ \bar L(\widehat x)-F_1(\widehat x)-F_2(\widehat x)-\bar k$ for some $m\in M$ then it holds: 
\begin{align*}
&\widehat y\in m+\bar L(\widehat x)-F_1(\widehat x)- F_2 (\widehat x+u) - K,\quad  \forall u\in U\\
\Longrightarrow\; & \widehat y-k\in m+\bar  L(\widehat x)-F_1(\widehat x)- F_2 (\widehat x+u) - K,\quad  \forall k\in \inte K,\; \forall u\in U \\
\Longrightarrow\; & (\widehat y-k,-u)\in \Delta,\quad  \forall k\in \inte K,\; \forall u\in U \qquad \textrm{(by \eqref{eq:19a})}\\
\Longrightarrow\;&(\widehat y-\inte K)\!\times\! (-U)\subset \Delta\\
\Longrightarrow\;& (\widehat y-\inte K)\!\times\! \{0_X\}  \subset \inte \Delta.
\end{align*}

$\bullet$ {\it Step 2.} {\it We prove that $(\bar y,0_X)\not\in \inte \Delta$. }  Indeed, if we assume on the contrary, then there exists   $V\in \mathcal{N}_Y(0_{Y})$ such that  $(\bar y+V)\!\times\! \{0_X\} \subset\Delta$.  Take $\bar{k}\in V\cap \inte K$,  one gets $(\bar y+\bar{k},0_X)\in \Delta$. 
This leads to exist $\bar x\in \dom F_1$ and $\bar x'   \in \dom F_2$  such that $\bar y+\bar{k}\in {\bar L}(\bar{x})-F_1(\bar x)-F_2(\bar x')-K$ and $\bar x=\bar x'$, which contradicts the assumption  $\bar y\in Y\setminus [M+(\bar L-F_1-F_2)(X)-\inte K]$.

$\bullet$ {\it Step 3.} {\it Applying  of the  convex separation theorem.} 
By  the convexity of $M$, $F_1$ and $F_2$, it is easy to check that $\Delta$ is a convex subset of $Y\times X$.  Moreover, it follows from Steps 1 and 2 that $\inte \Delta\ne \emptyset$ and $(\bar y,0_X)\not\in \inte \Delta$. Now,    the convex separation theorem (\cite[Theorem 3.4]{Rudin91}) applying  to the convex sets  $\{(\bar y,0_X)\}$  and  $\Delta$ yields the existence of    a nonzero functional $(y^\ast_0,u^\ast_0)\in Y^{*}\!\times\! X^\ast$ satisfying
\begin{equation}  \label{eq_336d}
	y^\ast_0(\bar y) < y^\ast_0(y) +u^\ast_0(u),\quad \forall (y,u)\in \inte \Delta,
\end{equation}
and consequently,  
\begin{equation}  \label{eq_337d}
	y^*_0(\bar y) \le y^*_0(y) + u^\ast_0(u),\quad \forall (y,u)\in \Delta.
\end{equation}
Next, we  show that
\begin{equation}  \label{eq_338d}
y^*_0(k^\prime)<0, \quad \forall k^\prime\in \inte K.
\end{equation}
Take $k^\prime  \in\inte K$.  According to Step 1, there is
$\widehat y \in Y$  such that $ ( \widehat y-\inte K)\!\times\! \{0_X\}  \subset \inte \Delta$. On the other hand,  by 
\cite[Lemma 2.1(i)]{CDLP20},  there is  $\mu>0$ such that  $\bar y-\mu k^\prime \in \widehat y-\inte K$. Hence,  
 $(\bar y-\mu k^\prime ,0_X)\in  \inte\Delta$, which, together with  \eqref{eq_336d},  leads to  $y^*_0(\bar y)< y^*_0(\bar y-\mu k^\prime)$, or $y^*_0(k^\prime )<0 $, and    \eqref{eq_338d} holds. 

$\bullet$   {\it Step 4.} {\it Define $L_1, L_2$ and verify \eqref{eq:15neww}.}  
Take $k_0\in \inte K$ such that $y^\ast_0(k_0)=-1$ (it is possible by \eqref{eq_338d}). We  
define 
$$ L_2(u)=u_0^\ast(u) k_0,\; \forall u\in X \quad \textrm{ and }\quad  L_1= \bar L- L_2.$$ 
It is easy to see that $ L_1, L_2\in \L(X,Y)$, $L_1+L_2=\bar L$, and  $y_0^{\ast }\circ  L_2=-u^\ast_0$.
Thus,  \eqref{eq_337d} can be rewritten as
$y^*_0(\bar y) \le y^*_0(y - L_2(u))$ for all $(y, u)\in \Delta,$
or equivalently, $y^{\ast }_0( y- L_2(u) -\bar y) \geq 0$ for all $ (y,u)\in \Delta.$
So, by   \eqref{eq_338d},   $y- L_2(u) -\bar y \not\in  \inte K $, yielding  
\begin{equation}  \label{abcd22}
	\bar y\notin y- L_2(u)-\inte K,\quad \forall (y,u)\in \Delta.
\end{equation}

Now, as   $( m+\bar L(x) - F_1(x)- F_2(x'), x-x') \in \Delta $ for all $m\in M$, $x\in \dom F_1$, and $x'\in \dom F_2$, it follows from \eqref{abcd22} that 
\begin{equation*}
	\bar y\notin M+ \bar L(x)\!-\! F_1(x)\!-\!F_2(x') \!-\!  L_2(x\!-\!x') \!-\!\inte K, \; \forall (m,x,x')\in M \times \dom F_1 \times \dom F_2, 
\end{equation*}
which is  \eqref{eq:15neww}  and the lemma has been proved. \qed

\section{ {Proof of the Lemma  \ref{pro:3.3_nwww}} }

For the proof of (i),  take $(L_i, U_i)\in \exepi F_i^\ast$, $i=1,2$,  and show  
 that $(L_1+ L_2, U_1\uplus U_2)\in \exepi (F_1+F_2)^\ast$, or equivalently,
\begin{equation}
\label{eq:14e}
(F_1+F_2)^\ast(L_1+ L_2)\preccurlyeq_K U_1\uplus U_2.
\end{equation}
On the one hand, for each $i=1,2$, as $( L_i, U_i)\in \exepi F_i^\ast$, it holds $F_i^\ast( L_i)\preccurlyeq_K U_i$, and hence,
$F^\ast_1( L_1)\uplus F^\ast_2(L_2) \preccurlyeq_K  U_1\uplus  U_2$ (see Proposition  \ref{prop_1ab}(v)). 
On the other hand, we have
$(F_1+F_2)^\ast( L_1+ L_2)=\wsup [( L_1+ L_2-F_1-F_2)(X)]$ and  
\begin{align*}
F^\ast_1( L_1)\uplus F^\ast_2( L_2)&=\wsup [F^\ast_1( L_1)+ F^\ast_2( L_2)]\\
&=\wsup[\wsup [( L_1-F_1)(X)]
+ \wsup [( L_2-F_2)(X)]]\\
&=\wsup[( L_1-F_1)(X)+ ( L_2-F_2)(X)]
\end{align*}
(the last equality follows from Proposition \ref{pro_decomp}(vi)).
It is easy to see that $( L_1+L_2-F_1-F_2)(X)\subset (L_1-F_1)(X)+ ( L_2-F_2)(X)$, and consequently, by Proposition \ref{prop_4gg}(iii), one gets 
$(F_1+F_2)^\ast( L_1+ L_2)\preccurlyeq_K F^\ast_1( L_1)\uplus F^\ast_2( L_2)$
 and \eqref{eq:14e} follows from transitive property of $\preccurlyeq_K$.

The assertion (ii) follows from (i) and Proposition \ref{rel_epi}. Concretely,  one has (see \eqref{eqbbb}) $\Psi(\exepi F_1^\ast \boxplus \exepi F_2^\ast)\subset\Psi( \exepi (F_1+F_2)^\ast)=\epi(F_1+F_2)^\ast$.  Lastly,  
(iii) follows  from (i), taking  \eqref{eqbbb} and \eqref{eqccc}  into account. \qed


\section{ {Proof of the Basic Lemma 2 (Lemma  \ref{thm:3.1aa})} }


$\bullet$ {\it Proof of {\rm (i)}.}  By Lemma \ref {pro:3.3_nwww}(ii), it  suffices  to show that  $\epi (F_1+F_2)^\ast\subset \Psi(\exepi F_1^\ast \boxplus \exepi F_2^\ast)$. Take $(\bar L,\bar y)\in \epi (F+I_A)^*$. Then,  by \eqref{epiF-star},
\begin{equation*}
\bar y\notin (\bar L - F_1-F_2)(X)-\inte K.
\end{equation*} 
Apply now the Basic lemma 1 (Lemma \ref{def_exepi}) to the case where $M=\{0_Y\}$, there exist $L_1,L_2\in \L(X,Y)$ such that $L_1+L_2=L$ and 
\begin{equation} \label{eqaaa}
\bar y\notin (L_1-F_1)(X) +(L_2-F_2)(X)-\inte K.
\end{equation}
This, together with Proposition \ref{pro_decomp}(i), yields 
\begin{equation}\label{eq:C2new}
\wsup \Big[(L_1-F_1)( \dom F_1)+(L_2-F_2)(\dom F_2)\Big]\ne\{+\infty_Y\}.
\end{equation}
 And then, according to (iii), and (vi) of  Proposition \ref{pro_decomp},
\begin{equation*}
\begin{aligned}
&(L_1-F_1)(X) +(L_2-F_2)(X)-\inte K\\ 
&= \wsup \Big[(L_1-F_1)( \dom F_1)+(L_2-F_2)(\dom F_2)\Big]-\inte K\\
&= \wsup \Big[\wsup( L_1\!-\!F_1)(\dom F)+\wsup ( L_2 -F_2)(\dom F_2)\Big]-\inte K\\
&= \wsup \Big[F_1^\ast( L_1)+ F_2^\ast( L_2)\Big]-\inte K  
= F_1^\ast( L_1) \uplus F_2^\ast(L_2)-\inte K.
\end{aligned}
\end{equation*}
Combine this to \eqref{eqaaa}, we obtain
$  \bar y\notin    F_1^\ast( L_1) \uplus F_2^\ast(L_2)-\inte K.   $

As a WS-sum  belongs to $\P_p(Y)^\infty$, \eqref{eq:C2new} leads to the fact that
$F_1^\ast( L_1) \uplus  F_2^\ast( L_2)\in \P_p(Y)$. Consequently,  according to Lemma \ref{pro_6hh}(i),  one gets $\bar y \in F_1^\ast(L_1) \uplus F_2^\ast(L_2)+K$. Hence, by Proposition \ref{prop_1ab}(vi), there exists $\bar V\in \P_p(Y)$ such that $\bar y \in F_1^\ast( L_1) \uplus \bar V$ and 
$ F_2^\ast( L_2)\preccurlyeq_K \bar V$. So, by taking  $\bar U=F_1^\ast( L_1)\uplus \bar V\in \P_p(Y)$, one gets $\bar y\in \bar U$ and \
\begin{align*}
(\bar L,\bar U)&=(L_1 + L_2,F_1^\ast( L_1)  \uplus \bar V)=( L_1,F_1^\ast(\bar L_1))\boxplus ( L_2,\bar V)\in \exepi F_1^\ast \boxplus \exepi F_2^\ast,
\end{align*}
and consequently, $(\bar L,\bar y)\in \Psi(\exepi F_1^\ast \boxplus \exepi F_2^\ast)$  and  {\rm (i)} has been proved.  
\bigskip

$\bullet$ {\it Proof of {\rm (ii)}.} Due to  Lemma \ref {pro:3.3_nwww}(iii),  it suffices to show that 
\begin{equation} \label{eqextraa}
\Psi \left(\exepi (F_1+F_2)^\ast  \nplus \exepi F_3^\ast\right)    \ \ \subset \ \  \Psi \left(\exepi F_1^\ast \boxplus \exepi F_2^\ast  \nplus \exepi F_3^\ast\right).   
\end{equation} 

$(\alpha)$ Take $(\bar L,\bar y)\in\Psi\left(\exepi (F_1+F_2)^\ast\boxplus \exepi F_3^\ast\right)$. Then, 
there are
 $(\tilde L,\tilde U)\in \exepi (F_1+F_2)^\ast$ and $(L_3,U_3)\in \exepi F_3^\ast$  
 such that 
\begin{equation}
\label{eq:15e}
 \tilde L+L_3= \bar L \quad\textrm{ and  }\quad\bar y\in \tilde U\uplus U_3.
\end{equation}
As $(\tilde L,\tilde  U)\in  \exepi (F_1+F_2)^\ast$ and $(L_3,U_3)\in \exepi F_3^\ast$, one gets    $(F_1+F_2)^\ast(\tilde L)\preccurlyeq_K\tilde U$ and  $F_3^\ast( L_3)\preccurlyeq_K U_3$, which, together with    Proposition \ref{prop_1ab}(v), yields 
\begin{equation}\label{eq:24bbis}
 (F_1+F_2)^\ast(\tilde L)\uplus F_3^\ast( L_3)  \preccurlyeq_K\tilde U\uplus U_3. 
\end{equation}

$(\beta)$ {Now as  $\bar y\in \tilde U\uplus U_3$ (see \eqref{eq:15e}), $\tilde U\uplus U_3\ne \{+\infty_Y\}$, we get from \eqref{eq:24bbis} that   $(F_1+F_2)^\ast(\tilde L)\uplus F_3^\ast( L_3)\ne \{+\infty_Y\}$, and hence,  $\tilde U\uplus U_3 \subset Y$ and $ (F_1+F_2)^\ast(\tilde L)\uplus F_3^\ast( L_3)\subset Y$. It now follows from  \eqref{eq:24bbis} and Proposition \ref{prop_4gg}(i) that  
\begin{equation}\label{eq:25nwew}
[(F_1+F_2)^\ast(\tilde L) \uplus F_3^\ast( L_3)-\inte K] \cap [ \tilde U\uplus U_3]=\emptyset.
\end{equation}}

\indent $(\gamma)$ It follows from  \eqref{eq:15e}, \eqref{eq:25nwew} that 
\begin{equation} \label{eqnewa} 
\bar y \not\in  (F_1+F_2)^\ast(\tilde L) \uplus F_3^\ast( L_3)-\inte K.
\end{equation}

On the other hand, from  Proposition \ref{pro_decomp} (vi) and (iii) that (see also \eqref{eq:15e}) 
\begin{equation*}
\label{eq:22cbis}
\begin{aligned}
&(F_1+F_2)^\ast(\tilde L)\uplus F_3^\ast( L_3)-\inte K
=\wsup[(F_1+F_2)^\ast(\tilde L)+F_3^\ast( L_3)]-\inte K\\
&=\wsup \Big[\wsup(\!\tilde L\!-\!F_1\!-\!F_2\!)(\!\dom F_1\cap \dom F_2\!)+\wsup(\! L_3\!-\!F_3\!)(\!\dom F_3\!)\Big]\!-\!\inte K\\
&=\wsup(\tilde L\!-\!F_1\!-\!F_2)(\dom F_1\cap \dom F_2)+\wsup( L_3\!-\!F_3)(\dom F_3)\!-\!\inte K\\
&=(\tilde L_1-F_1-F_2)(\dom F_1\cap\dom F_2)+(L_3-F_3)(\dom F_3)-\inte K.
\end{aligned}
\end{equation*}
This and    \eqref{eqnewa}  yields 
 \begin{eqnarray} 
 \bar y&\notin& (\tilde L_1-F_1-F_2)(\dom F_1\cap \dom F_2)+(L_3-F_3)(\dom F_3)-\inte K, \ \textrm{or},  \nonumber  \\
\label{eq:C9new}
 \bar y &\notin& (\tilde L_1-F_1-F_2)(X)+(L_3-F_3)(\dom F_3)-\inte K.
\end{eqnarray} 

\indent $(\delta)$ { By  Basic lemma 1 (apply to $M=(L_3-F_3)(\dom F_3)$), there exist $L_1,L_2\in\L(X,Y)$ such that $L_1+L_2=\tilde L$ (note  that, together with \eqref{eq:15e}, $L_1+L_2+L_3=\bar L$) and  
\begin{align*}
\bar y\notin (L_1-F_1)(X) +(L_2-F_2)(X)+(L_3-F_3)(\dom F_3)-\inte K.
\end{align*}}
A similar argument  as in the proof  of 
 \eqref{eq:22cbis} one has 
 \[
 (L_1-F_1)(X) +(L_2-F_2)(X)+(L_3-F_3)(\dom F_3)-\inte K =
F_1^\ast( L_1)\uplus F_2^\ast( L_2)\uplus F_3^\ast( L_3)-\inte K, 
\]
(note that \eqref{eq:C9new} ensures $\wsup[(L_1-F_1)(X) +(L_2-F_2)(X)+(L_3-F_3)(\dom F_3)]\ne \{+\infty_Y\}$)
which yields 
$\bar y\notin     F_1^\ast( L_1)\uplus F_2^\ast( L_2)\uplus F_3^\ast( L_3)-\inte K. $
{By the same argument as in the proof of \eqref{eq:lxab},
there exists $\bar V\in \P_p(Y)$ such that $y\in \bar V$ and 
$(\bar L,\bar V)
\in  \exepi F_1^\ast\boxplus\exepi F_2^\ast \boxplus \exepi F_3^\ast$
 which  means that $(\bar L,\bar y)\in\Psi\left(\exepi F_1^\ast\boxplus\exepi F_2^\ast \boxplus \exepi F_3^\ast\right)$ and \eqref{eqextraa} has been proved. }

\bigskip

$\bullet$ {\it Proof of {\rm (iii)}.}   { Firstly, note that   from  (ii) one  gets 
\begin{equation}  \label{eqthm31a}
\Psi\left(\exepi F_1^\ast\boxplus \exepi F_2^\ast\boxplus\exepi F_3^\ast \right)=\Psi\left(\exepi (F_1+F_2)^\ast\boxplus \exepi F_3^\ast\right). 
\end{equation}
\indent Assume now that $(C_0')$ holds. Then apply (i) to the two maps  $F_1+F_2$ and $F_3$ (play the roles of $F_1$ and $F_2$, respectively), one gets   $\Psi\left(\exepi (F_1^\ast+ F_2)^\ast\boxplus\exepi F_3^\ast \right)=\epi (F_1+F_2+F_3)^\ast$, which together with \eqref{eqthm31a},  proves  \eqref{eqthm31}.   }

{In the case when $(C_0'')$ holds,  one applies (i) to the mappings $F_3$ and $F_1+F_2$. The equalities in \eqref{eqthm31} then
follow by the similar argument as above. }
  \qed


\vskip-0.5cm



\begin{thebibliography}{99}


\bibitem{AB85} Aliprantis~CD, Burkinshaw~O. Positive operators. Orlando FL:
Academic Press;  1985


\bibitem{B12} {Bhatia~M. Higher order duality in vector optimization over cones. Optim. Lett. 2012;6:17-30}

\bibitem{Bol98} Bolintin\'eanu~S. Vector variational principles towards asymptotically well behaved vector convex functions. In:  Nguyen~VH, Strodiot~JJ, Tossings~P, editors. Lecture Notes in Econom and Math Systems 481. Berlin: Springer;  2000  



\bibitem{Bol01} Bolintin\'eanu~S. Vector Variational Principles; $\varepsilon$-Efficiency and Scalar Stationarity. J Convex Anal. 2001;8:71-85   


\bibitem{Bot2010} Bo\c{t}~RI. Conjugate Duality in Convex Optimization. Berlin: Springer-Verlag;  2010



\bibitem{BGW097nw}  Bo\c{t}~RI, Grad~SM, Wanka~G. A general approach for studying duality  in multiobjective  optimization. Math Meth Oper Res. 2007;65(3):417-444 

\bibitem{BGW09} Bo\c{t}~RI, Grad~SM,  Wanka~G. Duality in Vector
Optimization. Berlin: Springer; 2009


\bibitem{BGWMIA09} Bo\c{t}~RI, Grad~SM, Wanka~G. New regularity conditions for Lagrange and Fenchel--Lagrange duality in infinite dimensional spaces.  Math Inequal Appl.  2009;12(1):171-189













\bibitem{CDLP20}  C\'{a}novas~MJ, Dinh~N,  Long~DH, Parra~J. A new approach to strong duality for composite  vector optimization problems.   Optimization, 2020 (to appear).  \\
https://www.tandfonline.com/doi/full/10.1080/02331934.2020.1745796

\bibitem{DGLL17}  Dinh~N,  Goberna~MA, Long~DH, L\'opez~MA. New Farkas-type results for vector-valued function: A non-abstract approach. J Optim Theory Appl. 2019;182:4-29 

\bibitem{DGLMJOTA16}  Dinh~N, Goberna~MA, L\'{o}pez~MA, Mo~TH. 
Farkas-type results for vector-valued functions with applications. J Optim
Theory Appl. 2017;173:357-390 



 \bibitem{DGLM-Optim17}  Dinh~N, Goberna~MA, L\'{o}pez~MA, Mo~TH.   Robust optimization revisited via robust vector Farkas lemmas. Optimization. 2017;66:939-963  


\bibitem{DL-ACTA-2020} Dinh~N, Long~DH. Sectional convexity of epigraphs of conjugate mappings with applications to robust vector duality. 
Acta Mathematica Vietnamica, 45, 2020, 525 - 553. 

\bibitem{DL2017} Dinh~N, Long~DH. Complete characterizations of robust strong duality for  robust vector optimization problems. Vietnam J Math. 2018;46:293-328  

\bibitem{DM-12} Dinh~N, Mo~TH. Qualification conditions and Farkas-type results for systems involving composite functions. Vietnam J Math. 2012;40:407-437

\bibitem{DMVV-Siopt} Dinh~N, Mo~TH, Vallet~G, Volle~M. A unified  approach to robust Farkas-type results with applications to robust optimization problem. SIAM J Optim.  2017;27:1075-1101 

 \bibitem{DNV-08}  Dinh~N, Nghia~ TTA, Vallet~G.  A closedness condition and its applications to DC programs with convex constraints. Optimization. 2010;59:541-560







\bibitem{DVN-08} Dinh~N, Vallet~G, Nghia~TTA. Farkas-type results and duality for DC programs with convex constraints. 
J. Convex Anal. 2008;15:1-27

 \bibitem{DVV-14}  Dinh~N, Vallet~G, Volle~M. Functional inequalities
and theorems of the alternative involving composite functions.  J Global
Optim. 2014;59:837-863 

\bibitem{GP14} Grad~SM, Pop~EL. Vector duality for convex vector optimization problems by means of the quasi-interior of the ordering cone. Optimization. 2014;63:21-37

\bibitem{Jahn-vector}   {Jahn~J. Vector optimization.  Berlin: Springer; 2004 }









\bibitem{Jeya 1}  Jeyakumar~V.  The strong conical hull intersection property for convex programming.  Math Program. 2006;106:81-92     



\bibitem{JSDL05}  Jeyakumar~V,  Song~W,  Dinh~N, Lee~GM. Stable strong duality in convex optimization. Applied Mathematics Report AMR 05/22, University of New South Wales

\bibitem{Khan-Tammer-Zali} Khan~AA, Tammer~C,  Z\u{a}alinescu~C. Set-valued Optimization.  Berlin: Springer; 2015 




\bibitem{KLS89}  Krasnosel'skij~MA, Lifshits~JA, Sobolev~AV. 
 Positive Linear Systems. The Method of Positive Operators, Translated from
the Russian by J. Appell. Sigma Series in Applied Mathematics 5. Berlin:
Heldermann;  1989

\bibitem{Kuroiwa98}
Kuroiwa~D. The natural criteria in set-valued optimization. S\=urikaisekikenky\=usho
K\=oky\=uroku. 1998;1031:85-90  


\bibitem{LT07} {L\"{o}hne~A, Tammer~C. A new approach to duality in vector optimization. Optimization. 2007;56:221-239}



\bibitem{Luc} Luc~DT.  Theory of vector optimization. Berlin: Springer; 1989 

\bibitem{Tanino92} Tanino~T.  Conjugate duality in vector optimization. 
 J Math Anal Appl. 1992;167:84-97 











\bibitem{Rudin91} Rudin~W.  Functional Analysis, 2nd ed. New York:  McGraw-Hill; 1991




\bibitem{Z83} {Z\u{a}linescu~C. Duality for vectorial nonconvex optimization by convexification and applications. An Stiint Univ Al I Cuza Iasi Sect I a Mat. 1983;29(3):15-34}
\end{thebibliography}
\end{document}